\documentclass[11pt]{amsart}
\usepackage[T1]{fontenc}
\usepackage[a4paper,left=3.5cm,right=3.4cm,top=3.5cm,bottom=3.5cm]{geometry}

\usepackage{lmodern,amsmath,amssymb,amsthm,mathtools}
\usepackage{booktabs,longtable,array,microtype,float}
\usepackage{xcolor,listings,needspace,textcomp}
\usepackage{etoolbox}
\patchcmd{\section}{\scshape}{\bfseries}{}{}
\patchcmd{\subsubsection}{\itshape}{\bfseries}{}{}
\usepackage[
colorlinks=true,
linkcolor=blue,
citecolor=blue,
urlcolor=blue
]{hyperref}

\numberwithin{equation}{section}
\newtheorem{theorem}{Theorem}[section]
\newtheorem{proposition}[theorem]{Proposition}
\newtheorem{lemma}[theorem]{Lemma}

\theoremstyle{definition}

\theoremstyle{remark}
\newtheorem{remark}[theorem]{Remark}
\theoremstyle{plain}
\newtheorem{claim}{Claim}
\newtheorem{conjecture}[theorem]{Conjecture}
\allowdisplaybreaks[1]

\newcommand{\E}{\mathbb E}
\DeclareMathOperator{\Var}{Var}

\lstdefinelanguage{Maple}{
	morekeywords={proc,local,option,remember,for,from,to,do,end,if,then,
		else,return,error,and,or,not,in,restart},
	sensitive=true,morecomment=[l]{\#},morestring=[b]"}
\allowdisplaybreaks[4]

\title[Almkvist's conjecture]{Proof of Almkvist's conjecture on the  unimodality of partition polynomials}

\author{Jianxi Mao}
\address{(J. Mao) \newline \indent School of Mathematical Sciences, Dalian University of Technology, Dalian 116024, P. R. China}
\email{maojx@dlut.edu.cn}

\author{Wenle Shi}
\address{(W. Shi) \newline \indent School of Mathematical Sciences, Dalian University of Technology, Dalian 116024, P. R. China}
\email{shi-wenle@hotmail.com}

\author{Bao-Xuan Zhu$^{*}$}
\address{(B. Zhu) \newline \indent School of Mathematics and Statistics, Jiangsu Normal University, Xuzhou 221116, P. R. China}
\email{bxzhu@jsnu.edu.cn}
\thanks{This work was supported in part by the National Key R\&D Program of China (No. 2024YFA1013903) and National Natural Science Foundation of China (Nos. 12201100, 12471310).}



\begin{document}

\begin{abstract}
			For integers $r\ge2$ and $n\ge1$, let
$$
F_{r,n}(q)=\prod_{k=1}^{n}\frac{1-q^{rk}}{1-q^k}.
$$
The coefficient of $q^j$ in \(F_{r,n}(q)\) counts partitions of $j$ into parts of size at most $n$, each occurring at most $r-1$ times.
Hughes proved that $F_{2,n}(q)=\prod_{k=1}^{n}(1+q^k)$ is unimodal for every $n\ge 1$.
This result was reproved by Stanley using an algebraic approach and Odlyzko and Richmond using an analytic approach.
Almkvist conjectured that  $F_{r,n}(q)$ is unimodal in the following two cases:
every even $r$ and every $n\ge 1$; every odd $r$ and every $n\ge 11$.
He proved that this conjecture is true for $3\le r \le 20$ and $r=100,101$.
In this paper, we completely settle the conjecture.
\end{abstract}
\maketitle

\begingroup
\renewcommand{\thefootnote}{*}
\footnotetext{Corresponding author.}
\endgroup

\setcounter{tocdepth}{2}
\tableofcontents

\section{Introduction}
A finite real sequence $a_0,a_1,\ldots,a_m$ is \emph{unimodal} if there is an index $j$ such that
$$
a_0\le a_1\le\cdots\le a_{j-1}\le a_j\ge a_{j+1}\ge\cdots\ge a_m.
$$
A polynomial is unimodal when its coefficient sequence is unimodal.
The unimodality problems arise naturally in partition polynomials.
Cayley~\cite{Cay56} conjectured the unimodality of the Gaussian polynomial
$$
\genfrac{[}{]}{0pt}{}{n}{k}
=\prod_{j=1}^{k}\frac{1-q^{n-k+j}}{1-q^j},
$$
whose coefficient of $q^j$ counts partitions of $j$ into at most $k$ parts,
each of size at most $n-k$.
This was proved  by Sylvester~\cite{Syl78}
using semi-invariants of binary forms.

Since then, there has been an extensive literature on the unimodality of various kinds of partition polynomials~\cite{BW05,DJ23,DJ26,Ent68,GZ25,ZZ24}.
See, for example, 
Stanton~\cite{Sta90} and Stanley and Zanello~\cite{SZ15} for the unimodality of partition polynomials of Ferrers shapes; 
Chen, Ji and Zang~\cite{CJZ15, CJZ16} and Ji and Zang~\cite{JZ21}
for the unimodality of crank partition polynomials; 
Galashin and Lam~\cite{GL24} for 
$q,t$-unimodality of rational $q,t$-Catalan polynomials,
the generating polynomials of restricted partitions whose Young diagrams are contained in a fixed staircase shape.

Beyond partition polynomials, 
the study of unimodality has also attracted considerable attention
in combinatorics, algebra and geometry.
Huh~\cite{Huh12} proved that the coefficient  sequence of the chromatic polynomial,
taken in absolute value, is log-concave and hence unimodal,
which was conjectured by Read~\cite{Rea68}.
Elias and Williamson~\cite{EW21} proved the unimodality of the Kazhdan--Lusztig basis structure constants.
Park~\cite{Par25} established the
unimodality of linear syzygy Betti numbers associated with
sufficiently positive embeddings of smooth projective curves.
For further results on unimodality,
see, for instance,~\cite{BH20,Huh15,KS16,SW16}.

In this paper, we study the unimodality of
\begin{equation*}
	F_{r,n}(q)=\prod_{k=1}^{n}\frac{1-q^{rk}}{1-q^k} =\prod_{k=1}^{n}(1+q^k+\cdots+q^{(r-1)k}),
	\qquad \textrm{for }\, r\ge 2, n\ge 1.
\end{equation*}
The coefficient of $q^j$ in $F_{r,n}(q)$ counts partitions of $j$ into
parts at most $n$, each occurring at most $r-1$ times
\cite{Alm89}.
Hughes~\cite{Hug77} first proved that
$F_{2,n}(q)=\prod_{k=1}^{n}(1+q^k)$ is unimodal
by Lie algebraic methods.
This result was reproved by Stanley~\cite{Stanley80} using the hard Lefschetz theorem,
and by Odlyzko and Richmond~\cite{OR82} using an analytic method. 

In 1985, Almkvist proposed the following more general conjecture.

\begin{conjecture}[{\cite{Alm85}}]\label{con:main}
	Let $r\ge2$ and $n\ge1$ be integers. Then the polynomial $F_{r,n}(q)$ is
	unimodal in each of the following cases:
	\begin{enumerate}
		\item[\rm{(i)}] $r$ is even and $n\ge1$.
		\item[\rm{(ii)}] $r$ is odd and $n\ge11$.
	\end{enumerate}
\end{conjecture}

Almkvist~\cite{Alm87,Alm89} proved the conjecture for
$r\in\{3,4,\ldots,20,100,101\}$. 
Recently, Abdallah and McDaniel~\cite{AM25} studied
Almkvist's conjecture by interpreting $F_{r,n}(q)$ as the
Hilbert series of a nonstandard graded Artinian complete
intersection.
In particular, they proved that $F_{r,2}(q)$ is unimodal
if and only if $r\ge2$ is even. 

In this paper,  we completely prove the conjecture.
\begin{theorem}\label{thm:main}
	Let $r\ge2$ and $n\ge1$ be integers. Then
	$F_{r,n}(q)$
	is unimodal whenever $r$ is even or $n\ge11$.
\end{theorem}

It suffices to prove Theorem~\ref{thm:main} for $r\ge21$ due to Hughes~\cite{Hug77} and Almkvist~\cite{Alm87,Alm89}.
Let $A_{r,n}(j)$ be the coefficient of $q^j$ in $F_{r,n}(q)$.
We will investigate its first difference $B_{r,n}(j)=A_{r,n}(j)-A_{r,n}(j-1).$

The paper is organized as follows.
In Section~\ref{sec:differences}, we establish preliminary results on the first differences $B_{r,n}(j)$.
We derive an induction criterion for unimodality in terms of $B_{r,n}(j)$. 
In Section~\ref{sec:lowrows},  we prove Theorem~\ref{thm:main} for $n\le 11.$
We use polynomial representations of restricted partition functions~\cite{Gaj22,RodsethSellers2006}
to express $B_{r,n}(j)$ as a finite sum of polynomial values.
Then, we obtain the lower bounds for $B_{r,n}(j)$
using polynomial estimates.
Using a computer-assisted proof, we check finitely many
rational inequalities of the nonnegativity of the lower bounds, and hence derive the nonnegativity of $B_{r,n}(j)$
in the required ranges.
In Section~\ref{sec:fourier}, we prove Theorem~\ref{thm:main} for $n\ge 12.$
We use Fourier inversion to derive the integral representation of $B_{r,n}(j)$.
Then we split the integral into a main term and a tail term,
and show that the main term exceeds the absolute value
of the tail in the required ranges for all $r\ge21$. 
Together with the results of Section~\ref{sec:lowrows},
this completes the induction on $n$ starting at $n=11$.
Some computational verifications are included in the
appendices.

\section{Preliminaries}\label{sec:differences}
Let 
\begin{align}\label{eq:main}
	F_{r,n}(q)=\prod_{k=1}^{n}(1+q^k+\cdots+q^{(r-1)k}) =\sum_{j=0}^{D_{r,n}}A_{r,n}(j)q^j,
\end{align}
where the degree of $F_{r,n}(q)$ is
$D_{r,n}=(r-1){n(n+1)}/2$.
By convention, $F_{r,0}(q)=1$ and  $D_{r,0}=0$. 
Clearly,
$F_{1,n}(q)=1$
and $F_{r,n}(1)=r^n$.
It is trivial that $F_{r,n}(q)$ is unimodal when $r=1$ or $n=0$.
In the following, we suppose that $r\ge 2$ and $n\ge 1$.

Set $A_{r,n}(j)=0$ for
$j<0$ or $j>D_{r,n}$, and define the first difference of $A_{r,n}(j)$ by
\begin{align}\label{eq:B=A-A}
	B_{r,n}(j)=A_{r,n}(j)-A_{r,n}(j-1).
\end{align}
Let $[q^j] f(q)$ denote the coefficient of $q^j$ in the polynomial $f(q)$.
Then 
\begin{equation}\label{GF-B}
	B_{r,n}(j)=[q^j](1-q) F_{r,n}(q).
\end{equation}
It is immediate that $B_{r,n}(m)=0$ for $m<0$ and $B_{r,n}(0)=1$.

It follows from~\cite[Eq.~(2)]{AM25} that the coefficients
of $F_{r,n}(q)$ are symmetric, i.e., for every integer $j$,
\begin{equation}\label{fd:symmetry}
	A_{r,n}(j)=A_{r,n}(D_{r,n}-j).
\end{equation}
Define
$$
\mu_{r,n}:=\frac{D_{r,n}}{2}
=\frac{(r-1)n(n+1)}{4}.
$$
By \eqref{eq:B=A-A} and~\eqref{fd:symmetry}, we obtain the following result.

\begin{proposition}\label{fd:window}
	Let $r\ge 2$ and $n\ge 1$ be integers.
	Then $F_{r,n}(q)$ is unimodal if and only if
	$$
	B_{r,n}(j)\ge0
	\qquad\text{for }1\le j\le\left\lfloor\mu_{r,n}\right\rfloor.
	$$
\end{proposition}
The following result gives the nonnegativity of $B_{r,n}(j)$ 
by the unimodality of $F_{r,n-1}(q)$.
\begin{proposition}\label{fd:induction}
	Let $r\ge2$ and $n\ge1$ be integers.
	If $F_{r,n-1}(q)$ is unimodal, then
	$$
	B_{r,n}(j)\ge0
	\qquad\text{for }1\le j\le
	\left\lfloor\mu_{r,n-1}\right\rfloor.
	$$
\end{proposition}
\begin{proof}
	By the definition~\eqref{eq:main}, we have 
	\begin{equation}\label{eq:F-recur}
		F_{r,n}(q)
		=F_{r,n-1}(q)\sum_{\ell=0}^{r-1}q^{\ell n}.
	\end{equation}
	Comparing the coefficients of $q^j$ on both sides of~\eqref{eq:F-recur} yields
	\begin{equation*}
		A_{r,n}(j)
		=\sum_{\ell=0}^{r-1}A_{r,n-1}(j-\ell n).
	\end{equation*}
	Hence,
	\begin{align}\label{fd:recurrence}
		B_{r,n}(j)
		&=A_{r,n}(j)-A_{r,n}(j-1)\notag\\
		&=\sum_{\ell=0}^{r-1}
		\bigl(A_{r,n-1}(j-\ell n)
		-A_{r,n-1}(j-\ell n-1)\bigr)\notag\\
		&=\sum_{\ell=0}^{r-1}B_{r,n-1}(j-\ell n).
	\end{align}
	
	For $n=1,$ $F_{r,1}(q)=1+q+\cdots+q^{r-1}$.
	So $B_{r,n}(j)=1$ for $j=0$ and $B_{r,n}(j)=0$ for $1\le j\le r-1$.
	For $n\ge2$, suppose that $F_{r,n-1}(q)$ is unimodal.
	By Proposition~\ref{fd:window},
	$$
	B_{r,n-1}(j)\ge0
	\quad\text{for  }j\le
	\left\lfloor\mu_{r,n-1}\right\rfloor.
	$$
	Then $B_{r,n-1}(j-\ell n)\ge 0$ since $j-\ell n\le j$. 
	It follows from~\eqref{fd:recurrence}
	that $B_{r,n}(j)\ge0$ for
	$1\le j\le\lfloor\mu_{r,n-1}\rfloor$.
	This completes the proof.
\end{proof}

Combining Propositions~\ref{fd:window} and~\ref{fd:induction},
we obtain the following result.
\begin{theorem}\label{checkinginterval}
	Suppose that $F_{r,n-1}(q)$ is unimodal.
	If $B_{r,n}(j)\ge 0$  for
	\begin{equation*}
		\left\lfloor\mu_{r,n-1}\right\rfloor <j\le \left\lfloor\mu_{r,n}\right\rfloor,
	\end{equation*}
	then  $F_{r,n}(q)$ is unimodal.
\end{theorem}

We end this section by showing the nonnegativity  of $B_{r,n}(j)$
for $0\le j<2r$.

\begin{lemma}\label{fd:initial}
	Let $r\ge2$ and $n\ge3$ be integers. Then we have $B_{r,n}(j)\ge0$ for $0\le j< 2r.$
\end{lemma}

\begin{proof}
	By \eqref{GF-B}, we have
	$B_{r,n}(j)=[q^j](1-q)F_{r,n}(q).$
	Since $n\ge 3$, we have
	$$
	\begin{aligned}
		(1-q)F_{r,n}(q) =\frac{\prod_{k=1}^{n}(1-q^{rk})} {\prod_{k=2}^{n}(1-q^k)} =\frac{1-q^r}{(1-q^2)(1-q^3)}
		\left(\prod_{k=4}^{n}\frac{1}{1-q^k}\right) \left(\prod_{k=2}^{n}(1-q^{rk})\right),
	\end{aligned}
	$$
	where an empty product is interpreted as 1.
	
	Write $r=2u+3v$, where $u\ge 0$ and $v\in\{0,1\}$;
	take $v=0$ when $r$ is even and $v=1$  when $r$ is odd.
	Since $v\in\{0,1\}$,
	we have $1-q^{3v}=v(1-q^3)$ and
	\begin{align}\label{fd:num-decom}
		1-q^r
		=1-q^{2u+3v}
		=(1-q^{2u})+q^{2u}(1-q^{3v})
		=(1-q^2)\sum_{\ell=0}^{u-1}q^{2\ell}
		+vq^{2u}(1-q^3).
	\end{align}
	Dividing both sides of~\eqref{fd:num-decom}
	by $(1-q^2)(1-q^3)$, we obtain
	\begin{equation}\label{eq:non-expansion}
		\frac{1-q^r}{(1-q^2)(1-q^3)}
		=\frac{\displaystyle\sum_{\ell=0}^{u-1}q^{2\ell}}{1-q^3}
		+\frac{vq^{2u}}{1-q^2},
	\end{equation}
	where an empty sum is interpreted as 0.
	All coefficients on the right-hand side of~\eqref{eq:non-expansion}
	are nonnegative. Consequently, all coefficients of
	$$
	\frac{1-q^r}{(1-q^2)(1-q^3)}
	\prod_{k=4}^{n}\frac{1}{1-q^k}
	$$
	are nonnegative.
	
	Finally, every nonconstant term in
	$\prod_{k=2}^{n}(1-q^{rk})$
	has degree at least $2r$, and its constant term is $1$.  Hence, for $0\le j<2r$, we have
	$$
	B_{r,n}(j) = [q^j]\left( \frac{1-q^r}{(1-q^2)(1-q^3)} \prod_{k=4}^{n}\frac{1}{1-q^k} \right) \ge 0.
	$$
	This completes the proof.
\end{proof}

\section{Proof of Theorem~\ref{thm:main} for $r\geq 21$ and $n\leq 11$}\label{sec:lowrows}
In this section, 
we prove Theorem~\ref{thm:main} for $n\leq 11$. 
We first prove unimodality of $F_{r,n}(q)$ for even $r\ge22$ and $n=3,4$ in Section~\ref{sub-3.1}.
For the remaining cases $5\le n\le11$ for even $r\ge 22$ and $n=11$ for odd $r\ge 21$, 
we  first study the polynomial representation of the restricted partition functions 
and express $B_{r,n}(j)$ as a finite sum of polynomial values on each residue class in Section~\ref{sub-3.2}.
Then we derive the lower bound for $B_{r,n}(j)$ and prove the nonnegativity of this lower bound by finite verification in Section~\ref{sub-3.3}.
We complete the proof of Theorem~\ref{thm:main} for $n\le 11$ in Section~\ref{sub-3.4}.

\subsection{Proof of Theorem~\ref{thm:main} for $n=3$ and $n=4$}\label{sub-3.1}
We first deal with the case $n=3$.
\begin{proposition}\label{sc:three}
	For every integer $r\ge 2$, the polynomial $F_{r,3}(q)$ is
	unimodal.
\end{proposition}
\begin{proof}
	For $r\ge2$, since $D_{r,3}=6r-6$,
	Proposition~\ref{fd:window} shows that $F_{r,3}(q)$ is unimodal
	if and only if $B_{r,3}(j)\ge0$ for $1\le j\le3r-3$.
	The case $j<2r$ follows from Lemma~\ref{fd:initial}.
	It remains to prove $B_{r,3}(j)\ge0$ for $2r\le j\le3r-3$.
	
	For $m\ge0$, let
	\begin{equation}\label{c3m}
		c(m) =[q^m]\frac{1}{(1-q^2)(1-q^3)} =\#\{(x,y)\in\mathbb Z_{\ge0}^2:2x+3y=m\}.
	\end{equation}
	\begin{claim}\label{cla:n=3}
		We have
		\begin{equation}\label{eq:cm}
			c(m)=\left\lfloor\frac m2\right\rfloor -\left\lfloor\frac{m-1}{3}\right\rfloor.
		\end{equation}
	\end{claim}
	\noindent\textbf{Proof of Claim~\ref{cla:n=3}.}
	Let $m=2x+3y$ and $k=x+y$. Then $x=3k-m$ and $y=m-2k$.
	Since $(x,y)\in\mathbb Z_{\ge0}^2$,
	we have $m/3\le k\le m/2$ and $k\in \mathbb Z_{\ge0}.$
	Every $k$ corresponds to a pair $(x,y)$. Thus we prove~\eqref{eq:cm}.
	
	Fix $2r\le j\le3r-3$. Then
	\begin{align}\label{eq:Br3}
		B_{r,3}(j)&=[q^j](1-q)F_{r,3}(q)\notag\\
		&=[q^j]\frac{(1-q^r)(1-q^{2r})(1-q^{3r})} {(1-q^2)(1-q^3)}\notag\\
		&=[q^j]\frac{1-q^r-q^{2r}+q^{4r}+q^{5r}-q^{6r}} {(1-q^2)(1-q^3)}.
	\end{align}
	Since $j<3r$, it follows from~\eqref{eq:cm}
	and~\eqref{eq:Br3} that
	\begin{align}\label{eq:Br3-floor}
		B_{r,3}(j) =&c(j)-c(j-r)-c(j-2r)\notag\\
		=&\left(\left\lfloor\frac j2\right\rfloor
		-\left\lfloor\frac{j-1}{3}\right\rfloor\right)-\left(\left\lfloor\frac{j-r}{2}\right\rfloor
		-\left\lfloor\frac{j-r-1}{3}\right\rfloor\right)\notag\\
        &-\left(\left\lfloor\frac{j-2r}{2}\right\rfloor
		-\left\lfloor\frac{j-2r-1}{3}\right\rfloor\right)\notag\\
		=&r	-\left\lfloor\frac{j-1}{3}\right\rfloor-\left\lfloor\frac{j-r}{2}\right\rfloor
		+\left\lfloor\frac{j-r-1}{3}\right\rfloor +\left\lfloor\frac{j-2r-1}{3}\right\rfloor,
	\end{align}
	where we use $\left\lfloor j/2\right\rfloor-\left\lfloor{(j-2r)}/{2}\right\rfloor=r$
	in the last equality.
	
	Let $t \equiv r \pmod{3}$ and \(t\in\{0,1,2\}\).
	Then \(r=3s+t\).
	Set
	\(h=3r-j\). Since \(j\le 3r-3\), we have \(h\ge 3\).
	Note that 
	\begin{equation}\label{a-x}
		\left\lfloor a-x \right\rfloor
		= a-\left\lceil x \right\rceil, \qquad \textrm{for}\quad a\in\mathbb{Z}.
	\end{equation}
	Substituting \(r=3s+t\) and \(j=3r-h\) into~\eqref{eq:Br3-floor}, and using~\eqref{a-x},
	we obtain
	\begin{align}\label{br3-j}
		B_{r,3}(j)
		=&r-\left\lfloor\frac{3r-h-1}{3}\right\rfloor
		-\left\lfloor\frac{2r-h}{2}\right\rfloor
		+\left\lfloor\frac{2r-h-1}{3}\right\rfloor
		+\left\lfloor\frac{r-h-1}{3}\right\rfloor\nonumber\\
		=&r-\left(r-\left\lceil\frac{h+1}{3}\right\rceil\right)-\left(r-\left\lceil\frac h2\right\rceil\right)
		+\left(2s-\left\lceil\frac{h+1-2t}{3}\right\rceil\right)\nonumber\\
        &+\left(s-\left\lceil\frac{h+1-t}{3}\right\rceil\right)\nonumber\\
		=&-t+\left\lceil\frac{h+1}{3}\right\rceil+\left\lceil\frac h2\right\rceil
		-\left\lceil\frac{h+1-2t}{3}\right\rceil
		-\left\lceil\frac{h+1-t}{3}\right\rceil.
	\end{align}
	
	For \(t=0\), it is immediate by~\eqref{br3-j} that
	\begin{equation}\label{t=0}
		B_{r,3}(j)
		=\left\lceil\frac h2\right\rceil
		-\left\lceil\frac{h+1}{3}\right\rceil\ge 0,
	\end{equation}
	since $h/2\ge (h+1)/3$ for $h\ge 3$.
	
	For \(t=1,2\), using
	\(\lceil(h-3)/3\rceil=\lceil h/3\rceil-1\), by~\eqref{br3-j}, we obtain that 
	\begin{align}\label{t=1}
		B_{r,3}(j)
		&=\left\lceil\frac{h+1}{3}\right\rceil+\left\lceil\frac h2\right\rceil
		-\left\lceil\frac h3\right\rceil
		-\left\lceil\frac{h-1}{3}\right\rceil-1\nonumber\\
		&=2\left\lceil\frac{h+1}{3}\right\rceil+\left\lceil\frac h2\right\rceil
		-\left(\left\lceil\frac h3\right\rceil+
		\left\lceil\frac{h-1}{3}\right\rceil
		+\left\lceil\frac{h+1}{3}\right\rceil\right)-1\nonumber\\
		&=2\left(\left\lfloor\frac h3\right\rfloor+1\right)+\left\lceil\frac h2\right\rceil-\left(h+1\right)-1\nonumber\\
		&=2\left\lfloor\frac h3\right\rfloor
		-\left\lfloor\frac h2\right\rfloor,
	\end{align}
	where we used
	\[
	\left\lceil\frac{h-1}{3}\right\rceil
	+\left\lceil\frac h3\right\rceil
	+\left\lceil\frac{h+1}{3}\right\rceil=h+1,
	\qquad
	\left\lceil\frac{h+1}{3}\right\rceil
	=\left\lfloor\frac h3\right\rfloor+1.
	\]
	Recall that $h=3r-j\ge 3$.
	Write \(h=3a+b\), where
	\(a\ge 1\) and \(b\in\{0,1,2\}\). Since
	\[
	h=3a+b\le 3a+2\le 4a+1,
	\]
	we have
	\[
	\left\lfloor\frac{h}{2}\right\rfloor
	\le 2a
	=2\left\lfloor\frac{h}{3}\right\rfloor.
	\]
	Thus, by~\eqref{t=1}, we have $B_{r,3}(j)\ge0$ for $t=1,2$.
	Combining with~\eqref{t=0} for $t=0,$
	we complete the proof.
\end{proof}

Let $f(q)=a_0+a_1q+\cdots+a_dq^d$ be a polynomial of degree $d$.
We say $f(q)$ is \emph{symmetric} if $a_k=a_{d-k}$ for all $k$.
We introduce a result on unimodal polynomials that will be used in the proof of the case $n=4$.
\begin{lemma}\label{sc:closure}
	Suppose that $P(q)$ is a polynomial of odd degree $2d+1$.
	If $P(q)$ is symmetric and unimodal with all coefficients
	nonnegative, then so is
	$$
	Q(q)=(1+q)(1+q^3)P(q^2).
	$$
\end{lemma}
\begin{proof}
	Clearly, all coefficients of $Q(q)$ are nonnegative. 
	The symmetry of $P(q)$ gives $q^{4d+6}Q(q^{-1})=Q(q)$, 
	so the coefficients of $Q(q)$ are symmetric.
	Write
	$$
	P(q)=\sum_{k=0}^{2d+1}a_kq^k, \qquad Q(q)=\sum_{k=0}^{4d+6}b_k q^k,
	$$
	and set $a_k=0$ outside $0\le k\le2d+1$.
	Since $P(q)$ is symmetric and unimodal with nonnegative coefficients, we have 
	$$
	0\le a_0\le a_1\le\cdots\le a_d=a_{d+1}.
	$$

	Note that
	$Q(q)=(1+q+q^3+q^4)P(q^2).$
	Taking the coefficients of $q^{2k}$ and $q^{2k+1}$ on both sides yields
	$b_{2k}=a_k+a_{k-2}$ and $b_{2k+1}=a_k+a_{k-1}$.
	Thus
	\begin{align*}
		b_{2k}-b_{2k-1} &=a_k-a_{k-1}\ge0, \quad1\le k\le d+1,\\
		b_{2k+1}-b_{2k} &=a_{k-1}-a_{k-2}\ge0, \quad0\le k\le d+1.
	\end{align*}
	This gives $0\le b_0\le b_1\le\cdots\le b_{2d+3}.$ 
	By symmetry of $Q(q)$, we obtain the unimodality of $Q(q).$
	This completes the proof.
\end{proof}

We now deal with the case $n=4$.
\begin{proposition}\label{sc:small}
	For every even $r\ge2$, the polynomial $F_{r,4}(q)$ is unimodal.
\end{proposition}
\begin{proof}
	Let $r=2k$ and $k\ge 1$. Note that by \eqref{eq:main}
	\[
	F_{k,3}(q^2)
	=\frac{(1-q^{2k})(1-q^{4k})(1-q^{6k})}
	{(1-q^2)(1-q^4)(1-q^6)}.
	\]
	A direct calculation gives
	\begin{align}\label{F2K}
		F_{2k,4}(q)
		&=\frac{(1-q^{2k})(1-q^{4k})(1-q^{6k})(1-q^{8k})}
		{(1-q)(1-q^2)(1-q^3)(1-q^4)}\nonumber\\
		&=F_{k,3}(q^2)\,
		\frac{(1-q^{8k})(1-q^6)}{(1-q)(1-q^3)}\nonumber\\
		&=(1+q)(1+q^3)\,
		\frac{1-q^{8k}}{1-q^2}\,F_{k,3}(q^2).
	\end{align}
	Let $G(q)=\sum_{\ell=0}^{4k-1}q^{\ell}$. Then by~\eqref{F2K},
	$$
	F_{2k,4}(q) =(1+q)(1+q^3)G(q^2)F_{k,3}(q^2).
	$$  
	
	Sun, Wang and Zhang~\cite[Corollary 3.5]{SWZ15} showed that
	the product of two polynomials with nonnegative,
	symmetric, unimodal coefficients has the same property.
	The polynomial $G(q)$ has all coefficients equal to $1$.
	By Propositions~\ref{fd:window} and \ref{sc:three}, 
	the coefficients of $F_{k,3}(q)$ are nonnegative, symmetric, and unimodal.
	Thus, the coefficients of their product $G(q)F_{k,3}(q)$
	are all nonnegative, symmetric, and unimodal.
	Furthermore,
	$$
	\deg [G(q)F_{k,3}(q)] =(4k-1)+6(k-1) =10k-7,
	$$
	which is odd.
	Applying Lemma~\ref{sc:closure} to $G(q)F_{k,3}(q)$ yields
	$$
	F_{2k,4}(q)=(1+q)(1+q^3)G(q^2)F_{k,3}(q^2)
	$$
	with nonnegative, symmetric, unimodal coefficients.
	This completes the proof.
\end{proof}

In the following, we will prove Theorem~\ref{thm:main} for even $r\ge22$
and $5\le n\le11$, and for odd $r\ge21$ and $n=11$.
By Proposition~\ref{fd:window}, it suffices to prove
$B_{r,n}(j)\ge0$ for $1\le j\le\lfloor\mu_{r,n}\rfloor$.
We first express $B_{r,n}(j)$ as  a finite sum of polynomial values in $y$, where $y$ is an affine function of $j$.
We then derive rational lower bounds for $B_{r,n}(j)$
and finally show that these bounds are nonnegative.

\subsection{An exact expression for $B_{r,n}(j)$}\label{sub-3.2}

We will give an expression (see~\eqref{low:coefficient}) for $B_{r,n}(j)$  using a polynomial representation.
Assume that $n\ge 5$.
Write
\begin{align}
	\prod_{k=1}^{n}(1-q^k)&=\sum_{\ell=0}^{\frac{n(n+1)}2}u_{n,\ell}q^\ell,\label{eq:unl}\\
	\prod_{k=2}^{n}\frac{1}{1-q^k}&=\sum_{m\ge 0} c_n(m)q^m,\label{def:c}
\end{align}
where $c_n(m)=0$ for $m<0$.
Thus,
$$
\begin{aligned}
	(1-q)F_{r,n}(q)=\frac{\prod_{k=1}^{n}(1-q^{rk})}
	{\prod_{k=2}^{n}(1-q^k)}=\left(\sum_{\ell=0}^{\frac{n(n+1)}2}u_{n,\ell}q^{r\ell}\right)
	\left(\sum_{m\ge0}c_n(m)q^m\right).
\end{aligned}
$$
Comparing the coefficients of $q^j$ on both sides, by~\eqref{GF-B}, we obtain
\begin{equation}\label{B=UC}
	B_{r,n}(j) =[q^j](1-q)F_{r,n}(q) =\sum_{\ell=0}^{\frac{n(n+1)}2}u_{n,\ell}\,c_n(j-r\ell).
\end{equation}

We now introduce some properties of $c_n(m).$
Gajdzica~\cite[Lemma 3.1]{Gaj22} presented the recurrence relation of $c_n(m)$:
\begin{equation*}
	c_n(m)-c_n(m-n)=c_{n-1}(m).
\end{equation*}
This recurrence can be obtained by 	comparing coefficients of $q^m$ on both sides of the equation
$$
(1-q^n)\prod_{k=2}^{n}\frac{1}{1-q^k} = \prod_{k=2}^{n-1}\frac{1}{1-q^k}.
$$
Note that by~\eqref{def:c}, $c_n(m)$ counts partitions of $m$
whose parts belong to $\{2,3,\ldots,n\}$.
Therefore, 
$$
c_n(m)=\#\left\{(x_2,\ldots,x_n)\in\mathbb{Z}_{\ge0}^{n-1}:
2x_2+3x_3+\cdots+nx_n=m\right\}.
$$
Note that the notation $c(m)$ in~\eqref{c3m} corresponds to $c_n(m)$ for $n=3.$

We shall present a characterization of $c_n(m).$
Before that, we introduce a result of
R{\o}dseth and Sellers~\cite[Theorem~1 and Section~7]{RodsethSellers2006}.
They showed that, for a finite set \(A\) of positive integers, the number of restricted partitions \(p_A(m)\) is a polynomial in \(\mathbb{Q}[m]\) on each residue class modulo \(\operatorname{lcm}(A)\),
where \(\operatorname{lcm}(A)\) denotes the least common multiple of $A.$

\begin{lemma}\label{RS}
	Let $A$ be a finite set of $d\ge1$ positive integers
	with $\gcd(A)=1$, and let $p_A(m)$ denote the number
	of partitions of $m$ whose parts belong to $A$.
	For each integer $s$ with $0\le s<\operatorname{lcm}(A)$,
	there exists a polynomial $P_s(x)\in\mathbb{Q}[x]$
	of degree $d-1$ such that for $m\ge 0,$
	$$
	p_A(m)=P_s(m)\qquad \textrm{for}\quad  m\equiv s\pmod{\operatorname{lcm}(A)}.
	$$
	Moreover, the leading coefficient of $P_s(x)$ is
	$1/\bigl((d-1)!\prod_{a\in A}a\bigr)$.
\end{lemma}

We specialize these results to $A=\{2,\ldots,n\}$.
Then $p_A(m)=c_n(m)$ in Lemma~\ref{RS}.
Define the least common multiple of $\{2,3,\ldots,n\}$ by
$$
M_n=\operatorname{lcm}(2,\ldots,n) =\min\{m\in\mathbb Z_{>0}:k\mid m\text{ for every }2\le k\le n\}.
$$
Thus $M_n$ is an even integer and $M_n/k$ is an integer for every $2\le k\le n$.
The following lemma records the precise specialization of $c_n(m)$,
as an immediate consequence of Lemma~\ref{RS}.

\begin{lemma}\label{lemma-key}
	For each integer $s$ with $0\le s<M_n$, there exists
	a polynomial $P_s(x)\in\mathbb{Q}[x]$ of degree $n-2$
	such that for $m\ge0$, we have
	$$
	c_n(m)=P_s(m)\qquad \textrm{for}\quad m\equiv s\pmod{M_n}.
	$$
	Every $P_s(x)$ has leading coefficient
	$1/(n!(n-2)!)$ and is uniquely determined by
	$$
	P_s(s+M_nt)=c_n(s+M_nt), \qquad t=0,1,\ldots,n-2.
	$$
\end{lemma}
\begin{proof}
	Apply Lemma~\ref{RS} to $A=\{2,\ldots,n\}$.
	Since $\gcd(A)=1$ and $c_n(m)=p_A(m)$, for each
	$0\le s<M_n$ there exists a polynomial
	$P_s(x)\in\mathbb{Q}[x]$ such that
	$$
	c_n(m)=P_s(m)
	$$
	for $m\ge0$ with $m\equiv s\pmod{M_n}$.
	Its degree is $n-2$, and its leading coefficient is
	$1/(n!(n-2)!)$.
	
	Since $s+M_nt\ge0$ and $s+M_nt\equiv s\pmod{M_n}$,
	we have
	$$
	P_s(s+M_nt)=c_n(s+M_nt),
	\qquad t=0,1,\ldots,n-2.
	$$
	A polynomial of degree at most $n-2$ is uniquely determined
	by its values at $n-1$ distinct points.
	Therefore, the values $c_n(s+M_nt)$ for $t=0,1,\ldots,n-2$
	uniquely determine $P_s(x)$.
	This completes the proof.
\end{proof}

For the subsequent coefficient estimates, define
\begin{equation}\label{R-P}
	R_s(x)=P_s\left(x-\frac{n(n+1)-2}{4}\right)
	\qquad \textrm{for}\quad 0\le s<M_n.
\end{equation}
By Lemma~\ref{lemma-key},  we have 
\begin{equation}\label{pol-sub}
	c_n(m)=R_s\left(m+\frac{n(n+1)-2}{4}\right)
\end{equation}
for $m\ge0$ with $m\equiv s\pmod{M_n}$.
Each $R_s(x)$ has degree $n-2$ and leading coefficient
$1/(n!(n-2)!)$.

We now extend the identity~\eqref{pol-sub} from $m\ge0$
to negative integers $m$ satisfying 
\begin{equation}\label{range3.2}
	-(n(n+1)-2)/2<m<0\quad \textrm{and}\quad m\equiv s\pmod{M_n}.
\end{equation}
Since $c_n(m)=0$ for $m<0$ by definition, it suffices
to show that the right-hand side of~\eqref{pol-sub}
vanishes in the range~\eqref{range3.2}.
The following is a consequence of~\cite[Theorem~3]{RodsethSellers2006}.
For the sake of completeness, we present a proof.
\begin{lemma}\label{low:negative-extension}
	For every negative integer $m$ with $-\frac{n(n+1)-2}{2}<m<0$,
	let $m\equiv s\pmod{M_n}$ and $0\le s <M_n$.
	Then
	\begin{equation}\label{lemma36}
		R_s\left(m+\frac{n(n+1)-2}{4}\right)=c_n(m)=0.
	\end{equation}
\end{lemma}
\begin{proof}
	For every integer $m\in \mathbb{Z}$, let 	$0\le s <M_n$ and  $m\equiv s\pmod{M_n}$.  Define
	\begin{equation}\label{low:extension-definition}
		\widetilde c_n(m)
		=R_s\!\left(m+\frac{n(n+1)-2}{4}\right),
	\end{equation}
	where $\widetilde c_n(m)$ is a polynomial in $m$ on $s+M_n\,\mathbb{Z}$ with degree $n-2$.
	Thus,  for $m\ge 0$, we have 
	\begin{align}\label{low:extension-agreement}
		\widetilde c_n(m)=c_n(m).
	\end{align}
	Let
	$$
	d=\frac{n(n+1)-2}{2},\qquad \prod_{k=2}^{n}(1-q^k)=\sum_{j=0}^{d}\gamma_jq^j, \qquad
	\gamma_{d}=(-1)^{n-1}.
	$$
	By~\eqref{def:c}, we have 
	\begin{equation}\label{low:generating-identity}
		\prod_{k=2}^{n}(1-q^k)\sum_{m\ge0}c_n(m)q^m
		=\prod_{k=2}^{n}(1-q^k)\prod_{k=2}^{n}\frac{1}{1-q^k}
		=1.
	\end{equation}
	Comparing the coefficients of $q^m$ on both sides
	of~\eqref{low:generating-identity} for $m\ge1$, we obtain
	\begin{align}\label{low:coefficient-recurrence}
		\sum_{j=0}^{d}\gamma_j c_n(m-j)=0.
	\end{align}
	For $m\in \mathbb{Z}$, define
	$$
	H(m):=\sum_{j=0}^{d} \gamma_j \widetilde c_n(m-j).
	$$
	It follows immediately from~\eqref{low:extension-definition}
	that $H(m)$ is a polynomial in $m$ of degree at most $n-2$
	on each residue class $s+M_n\mathbb{Z}$. 
	For every $m\ge d,$ since $c_n(m-j)=\widetilde c_n(m-j)$, we have
	$$
	H(m)=\sum_{j=0}^{d}\gamma_j c_n(m-j)=0.
	$$
	For each fixed $s$ with $0\le s<M_n$, $H(s+M_nt)$
	is a polynomial in $t$.
	Since $H(m)=0$ for $m\ge d$, this polynomial vanishes
	whenever $s+M_nt\ge d$.
	It therefore has infinitely many integer zeros and is
	identically zero.
	Every integer $m$ can be written as $s+M_nt$ with
	$0\le s<M_n$ and $t\in\mathbb{Z}$.
	Hence, for every integer $m$, we have
	\begin{equation}\label{low:H-zero}
		H(m)=\sum_{j=0}^{d}\gamma_j\widetilde c_n(m-j)=0.
	\end{equation}
	By~\eqref{low:extension-agreement},
	\eqref{low:coefficient-recurrence}, and~\eqref{low:H-zero},
	we have
	\begin{align*}
		0=H(d-1)
		=\sum_{j=0}^{d}\gamma_j\widetilde c_n(d-1-j)
		=\sum_{j=0}^{d-1}\gamma_j c_n(d-1-j)
		+\gamma_d\widetilde c_n(-1)
		=\gamma_d\widetilde c_n(-1).
	\end{align*}
	Since $\gamma_d\ne0$, we obtain $\widetilde c_n(-1)=0$.
	Repeating this argument with $m=d-2,\ldots,1$ successively gives
	$$
	\widetilde c_n(-1)=\widetilde c_n(-2)
	=\cdots=\widetilde c_n(1-d)=0.
	$$
	By~\eqref{low:extension-definition} and
	$d=(n(n+1)-2)/2$, we obtain
	$$
	R_s\!\left(m+\frac{n(n+1)-2}{4}\right)=0
	$$
	for every integer $m$ with $-(n(n+1)-2)/2<m<0$,
	where $0\le s<M_n$ and $m\equiv s\pmod{M_n}$.
	This completes the proof.
\end{proof}

Now we give the desired expression for $B_{r,n}(j).$
Recall that, by~\eqref{B=UC},
\begin{equation*}
	B_{r,n}(j) =\sum_{\ell=0}^{\frac{n(n+1)}2}u_{n,\ell}\,c_n(j-r\ell).
\end{equation*}
By~\eqref{pol-sub} and~\eqref{lemma36}, for $m=j-r\ell>-(n(n+1)-2)/2$, we have 
\begin{equation}\label{section3last}
	c_n(j-r\ell) = R_{s_\ell}\left(j-r\ell+\frac{n(n+1)-2}{4}\right),
\end{equation}
where 
\begin{equation}\label{smod}
	s_\ell\equiv j-r\ell\pmod{M_n}.
\end{equation}
For each $j\ge 0$, set 
\begin{equation*}
	y=\frac{j+(n(n+1)-2)/4}{r}.
\end{equation*}
Then $r(y-\ell)=j-r\ell+(n(n+1)-2)/4$.
Hence,  for $j-r\ell>-(n(n+1)-2)/2$, we have
$$
c_n(j-r\ell) = R_{s_\ell}\left(j-r\ell+\frac{n(n+1)-2}{4}\right) = R_{s_\ell}\bigl(r(y-\ell)\bigr).
$$
Now we consider the range of $\ell$:
\begin{itemize}
	\item[\rm(i)] If $\ell\ge y$, then
	$j-r\ell\le-(n(n+1)-2)/4<0$.
	By definition~\eqref{def:c},
	$c_n(j-r\ell)=0$.
	\item[\rm(ii)] If $\ell<y$, then
	$j-r\ell>-(n(n+1)-2)/4>-(n(n+1)-2)/2$.
	By~\eqref{section3last},
	$c_n(j-r\ell)= R_{s_\ell}\bigl(r(y-\ell)\bigr)$.
\end{itemize}
Thus,
\begin{equation}\label{low:coefficient}
	B_{r,n}(j)= \sum_{\substack{0\le \ell\le \frac{n(n+1)}2\\\ell<y}}
	u_{n,\ell}R_{s_\ell}\bigl(r(y-\ell)\bigr) \qquad \textrm{for }\, y=\frac{j+(n(n+1)-2)/4}{r}.
\end{equation}
This gives the desired expression for $B_{r,n}(j)$.
We will use~\eqref{low:coefficient} to obtain a lower bound
for $B_{r,n}(j)$.



\subsection{Interval estimate of $B_{r,n}(j)$ for $5\le n\le 11$}\label{sub-3.3}
In this subsection, we will  give the lower bound for $B_{r,n}(j)$ and derive its nonnegativity in Proposition~\ref{sc-B},
which will be used in the proof of Theorem~\ref{thm:main} for $5\le n\le 11$.
\begin{proposition}\label{sc-B}
	The following statements hold.
	\begin{enumerate}
		\item[\rm(i)]
		For every even integer $r\ge22$ and every integer $5\le n\le11$,
		we have
		$B_{r,n}(j)\ge0$
		for
		\begin{equation}\label{even511}
			\left\lfloor\frac{(r-1)n(n-1)}4\right\rfloor \le j\le \left\lfloor\frac{(r-1)n(n+1)}4\right\rfloor.
		\end{equation}
		\item[\rm(ii)]
		For every odd integer $r\ge21$ and $n=11$, we have
		$B_{r,11}(j)\ge0$
		for
		\begin{equation}\label{odd11}
			9r\le j\le33(r-1).
		\end{equation}
	\end{enumerate}
\end{proposition}

We use the representation~\eqref{low:coefficient} of $B_{r,n}(j)$ to prove Proposition~\ref{sc-B}.
First, 
we convert the intervals~\eqref{even511} and~\eqref{odd11} of $j$ into those of $y$ by~\eqref{low:coefficient},
and divide them into desired subintervals.
We will estimate $B_{r,n}(j)$ in each subinterval.
We then decompose $B_{r,n}(j)$ into two terms.
After estimating the two terms separately, we complete the proof
by verifying finitely many rational inequalities.

\medskip
\noindent
\subsubsection{Subdivision of the $y$-ranges}
By~\eqref{low:coefficient}, we have
$$
y=\frac{j+(n(n+1)-2)/4}{r}.
$$
By~\eqref{even511} and~\eqref{odd11},
a direct calculation shows that for case~\rm(i) with even $r$ and $5\le n\le 11,$
\begin{equation}\label{range:y}
	y\in \left[\frac{n(n-1)}4, \frac{n(n+1)}4\right];
\end{equation}
and  for case~\rm(ii) with odd $r$ and $n=11,$
$$y\in [9,33].$$
Both intervals are independent of $r$.

We split each interval at all its interior integers.
For example, for case (i) with even $r$ and $n=7$, the interval in~\eqref{range:y} is
$[\frac{n(n-1)}4,\frac{n(n+1)}4]=[10.5,14]$.
We divide it into four subintervals
$[10.5,11], [11,12], [12,13], [13,14].$
The odd case uses the 24 subintervals
$[9,10],[10,11],\ldots,[32,33]$. In both cases, the last subinterval is the rightmost one,
whose right endpoint is $n(n+1)/4$.

For each subinterval $I=[a,b]$, put
$$
\mathcal J_I= \left\{\ell\in\mathbb Z_{\ge0}:\ell<\frac{a+b}{2}\right\}.
$$
If $y\in I$ and $y\ne a$ (i.e., $y$ is not the left endpoint of $I$), then
we have $\{\ell\in \mathbb{Z}_{\ge0}:\ell<y\}=\mathcal J_I.$
This still holds when $y=a$ and $a$ is not an integer.

If $y=a\in\mathbb{Z}$ corresponds to an integer coefficient index $j$, then $\mathcal J_I$ contains one
additional index, $\ell=a$.
By the expression of $y$ in~\eqref{low:coefficient}, $y=a$ implies that
$$
j-ra=-\frac{n(n+1)-2}{4}.
$$
Recall that~\eqref{smod} shows that $s_\ell\equiv j-r\ell\pmod{M_n}$. This
gives
$$
s_a+\frac{n(n+1)-2}{4}\equiv 0\pmod{M_n}.
$$
Taking $m=-{(n(n+1)-2)}/{4}$ in~\eqref{lemma36}, we have 
$R_{s_a}(0)=0.$
Consequently, \eqref{low:coefficient} becomes
\begin{equation}\label{low:fixed-sum}
	B_{r,n}(j)=\sum_{\ell\in\mathcal J_I} u_{n,\ell}R_{s_\ell}\bigl(r(y-\ell)\bigr), \qquad y\in I,
\end{equation}
where $0\le s_\ell<M_n$ and $s_\ell\equiv j-r\ell\pmod{M_n}$.
Clearly, every index in $\mathcal J_I$ is less than $n(n+1)/2$ since the right endpoint $b\le n(n+1)/4$.

\medskip
\noindent
\subsubsection{A decomposition of $B_{r,n}(j)$}
By Lemma~\ref{lemma-key} and Equation~\eqref{R-P}, $R_s(x)$ is a polynomial
of degree $n-2$ with leading coefficient $1/(n!(n-2)!)$.
Then we have
\begin{equation}\label{direct:Rs-expansion}
	R_s(x)=\frac{x^{n-2}}{n!(n-2)!}
	+\sum_{k=0}^{n-3}a_{k,s}x^k,
\end{equation}
where $0\le s<M_n$.
Combining~\eqref{low:fixed-sum} with~\eqref{direct:Rs-expansion}, we obtain
\begin{align}\label{direct:collected-sum}
	B_{r,n}(j)
	&=\sum_{\ell\in\mathcal J_I}
	u_{n,\ell}R_{s_\ell}\bigl(r(y-\ell)\bigr)\notag\\
	&=\frac{r^{n-2}}{n!(n-2)!}
	\sum_{\ell\in\mathcal J_I}u_{n,\ell}(y-\ell)^{n-2}
	+\sum_{k=0}^{n-3}r^k
	\sum_{\ell\in\mathcal J_I}
	u_{n,\ell}a_{k,s_\ell}(y-\ell)^k.
\end{align}
For each subinterval $I=[a,b]$, define
\begin{equation}\label{eq:Ly-sum}
	L(y)=\frac1{n!(n-2)!}
	\sum_{\ell\in\mathcal J_I}u_{n,\ell}(y-\ell)^{n-2}.
\end{equation}
Recall that  $0\le s_\ell<M_n$. By~\eqref{smod},  for a fixed $r$, $s_\ell$ is dependent on $j$ and $\ell$:
\begin{equation}\label{smod1}
	s_\ell\equiv j-r\ell\pmod{M_n}.
\end{equation}
Then for $0\le k\le n-3$, we define
\begin{equation}\label{direct:H-definition}
	H_{I,k}^{r,j}(y)=\sum_{\ell\in\mathcal J_I}
	u_{n,\ell}a_{k,s_\ell}(y-\ell)^k.
\end{equation}
Thus by~\eqref{direct:collected-sum},~\eqref{eq:Ly-sum} and~\eqref{direct:H-definition}, it follows that
\begin{equation}\label{low:main-error}
	B_{r,n}(j)=r^{n-2}L(y)+\sum_{k=0}^{n-3}r^kH_{I,k}^{r,j}(y).
\end{equation}
We will estimate the leading term $r^{n-2}L(y)$ and the lower-degree terms
$\sum_{k=0}^{n-3}r^kH_{I,k}^{r,j}(y)$ on each subinterval defined in Section 3.3.1.

\subsubsection{Lower bounds for $B_{r,n}(j)$ on each subinterval}
For each subinterval $I$ defined in Section 3.3.1, we will prove $B_{r,n}(j)\ge0$
by estimating the two terms
in~\eqref{low:main-error} in two steps. 
In Step~1, we obtain a lower bound for the leading term
$r^{n-2}L(y)$ (see~\eqref{low:leading-bound}).
In Step~2, we will get an upper bound for the lower-degree terms
$\left|\sum_{k=0}^{n-3}r^kH_{I,k}^{r,j}(y)\right|$
(see~\eqref{direct:lower-degree-absolute-bound}).
Combining~\eqref{low:leading-bound}
and~\eqref{direct:lower-degree-absolute-bound} will give a lower bound for
$B_{r,n}(j)$.

We introduce a method for estimating polynomials.
See~\cite[Section~5]{MKC09} for more details.
For a polynomial $P(x)\in\mathbb Q[x]$ and $x\in[a,b]$, set
$x=(a+b)/2+(b-a)z/2$ and write
$$
P\left(\frac{a+b}{2}+\frac{b-a}{2}z\right)
=\sum_{m=0}^{d}p_m z^m,
\qquad \textrm{for} \quad -1\le z\le 1.
$$
For an interval $I=[a,b]$ and a polynomial $P$, let
$$ E_{I,P}=p_0-\sum_{m=1}^{d}|p_m|.$$
Then  we have
\begin{equation}\label{eq:fIE}
	P(x)=p_0+\sum_{m=1}^{d}p_m z^m
	\ge p_0-\sum_{m=1}^{d}|p_m z^m|
	\ge E_{I,P}.
\end{equation}

\noindent\textbf{Step 1: A lower bound for the leading term $r^{n-2}L(y)$.}
To obtain a lower bound for $r^{n-2}L(y)$, we consider two cases
for the intervals $[a,b]$: $b\ne n(n+1)/4$ and $b=n(n+1)/4$.

\emph{Case 1: $b\ne n(n+1)/4$.}
By~\eqref{eq:fIE}, for $y\in[a,b]$, we have
\begin{equation}\label{eq:Ly-bound1}
	r^{n-2}L(y)\ge r^{n-2} E_{I,L}.
\end{equation}

\emph{Case 2: $b=n(n+1)/4$.}
When $5\le n\le11$, we have
\begin{equation}\label{eq:L-center-zero}
	L\left(\frac{n(n+1)}4\right)=0.
\end{equation}
We prove~\eqref{eq:L-center-zero} in Appendix~\ref{app:center-zero}.
Hence there exists a polynomial $\widehat L(y)\in\mathbb Q[y]$ such that $L(y)=\left(n(n+1)/4-y\right)\widehat L(y).$
For $y\in[a,n(n+1)/4]$, it follows from~\eqref{eq:fIE} that
\begin{equation}\label{eq:Ly-bound2}
	r^{n-2}L(y)\ge
	r^{n-2}\left(\frac{n(n+1)}4-y\right)\,E_{[a,n(n+1)/4],\widehat L}.
\end{equation}
The two cases give the following lower bound for $y\in[a,b]$:
\begin{equation}\label{low:leading-bound}
	r^{n-2}L(y)\ge
	\begin{cases}
		r^{n-2} \, E_{I,L}, & b\ne n(n+1)/4,\\[3pt]
		r^{n-2}\bigl(n(n+1)/4-y\bigr)\,E_{[a,n(n+1)/4],\widehat L},
		& b=n(n+1)/4.
	\end{cases}
\end{equation}

\noindent\textbf{Step 2: An upper bound for the lower-degree terms 
	$\left|\sum_{k=0}^{n-3}r^kH_{I,k}^{r,j}(y)\right|$.}
Fix $I=[a,b]$ and $y\in I$, and write $y=(a+b)/2+(b-a)z/2$.
Then $|z|\le1$. 
For each $r,j$ and $0\le k\le n-3$, write
\begin{align}\label{direct:H-expansion}
	H_{I,k}^{r,j}\left(\frac{a+b}{2}+\frac{b-a}{2}z\right)
	=\sum_{m=0}^{k}h_{k,m}^{r,j}z^m.
\end{align}
Then by~\eqref{direct:H-expansion},
we have
\begin{align}\label{eq:Hh}
	\left|H_{I,k}^{r,j}(y)\right|
	=\left|\sum_{m=0}^{k}h_{k,m}^{r,j}z^m\right|
	\le\sum_{m=0}^{k}|h_{k,m}^{r,j}z^m|
	\le\sum_{m=0}^{k}|h_{k,m}^{r,j}|.
\end{align}

Recall that by~\eqref{smod1}, $s_\ell\equiv j-r\ell\pmod{M_n}.$
Then
\begin{align*}
	&j-(r+M_n)\ell\equiv j-r\ell \equiv s_\ell\pmod{M_n},\\
	&(j+M_n)-r\ell\equiv j-r\ell \equiv s_\ell \pmod{M_n}.
\end{align*}
Thus, by~\eqref{direct:H-definition},
we obtain
\begin{equation}\label{direct:H-periodicity}
	H_{I,k}^{r+M_n,\,j}(x)=H_{I,k}^{r,j}(x),
	\qquad
	H_{I,k}^{r,\,j+M_n}(x)=H_{I,k}^{r,j}(x).
\end{equation}
By~\eqref{direct:H-periodicity}, it suffices to consider
$21\le r\le20+M_n$ and $0\le j<M_n$.

We will derive an upper bound for $\left|\sum_{k=0}^{n-3}r^kH_{I,k}^{r,j}(y)\right|$
according to the parity of $r$.
Then we derive the lower bound for $B_{r,n}(j)$ for the even case and odd case separately.
Write $r\equiv r'\pmod2$, where $r'\in\{0,1\}$.
For $0\le k\le n-3$, define
\begin{equation}\label{direct:D-positive}
	D_{I,k}^{r'}
	=\max_{\substack{21\le r\le20+M_n\\
			0\le j<M_n\\
			r\equiv r'\pmod2}}
	\sum_{m=0}^{k}|h_{k,m}^{r,j}|.
\end{equation}

By~\eqref{eq:Hh},~\eqref{direct:H-periodicity} and~\eqref{direct:D-positive}, we obtain
\begin{equation*}
	\left|H_{I,k}^{r,j}(y)\right|
	\le D_{I,k}^{r'}.
\end{equation*}
Since $r>0$, it follows that
\begin{equation}\label{direct:lower-degree-absolute-bound}
	\left|\sum_{k=0}^{n-3}r^kH_{I,k}^{r,j}(y)\right|
	\le\sum_{k=0}^{n-3}r^k
	\left|H_{I,k}^{r,j}(y)\right|
	\le\sum_{k=0}^{n-3}D_{I,k}^{r'}r^k.
\end{equation}

\subsubsection{Finite verification and proof of Proposition~\ref{sc-B}}
We now use~\eqref{low:leading-bound} and~\eqref{direct:lower-degree-absolute-bound}
to derive and verify $B_{r,n}(j)\ge 0$ on each subinterval.

\emph{Proof of Proposition~\ref{sc-B}.}
Fix a subinterval $I=[a,b]$ and the required parity $r'$ of $r$.
By~\eqref{low:main-error} and~\eqref{direct:lower-degree-absolute-bound}, we have
\begin{equation}\label{direct:B-lower-bound}
	B_{r,n}(j)\ge r^{n-2}L(y)
	-\sum_{k=0}^{n-3}D_{I,k}^{r'}r^k.
\end{equation}
We consider the cases $b\ne n(n+1)/4$ and $b=n(n+1)/4$ as follows.

\emph{Case 1: $b\ne n(n+1)/4$.} Applying~\eqref{eq:Ly-bound1} to~\eqref{direct:B-lower-bound}, we obtain
\begin{align*}
	B_{r,n}(j)
	&\ge r^{n-2}E_{I,L}-\sum_{k=0}^{n-3}D_{I,k}^{r'}r^k\\
	&=r^{n-2}\left(E_{I,L}-
	\sum_{k=0}^{n-3}D_{I,k}^{r'}r^{k+2-n}\right)\\
	&\ge r^{n-2}\left(E_{I,L}-
	\sum_{k=0}^{n-3}D_{I,k}^{r'}21^{k+2-n}\right).
\end{align*}
The last inequality uses $r\ge21$, $D_{I,k}^{r'}\ge0$,
and $k+2-n<0$.
We will prove 
\begin{equation}\label{low:noncentral-cert}
	E_{I,L}-\sum_{k=0}^{n-3}D_{I,k}^{r'}21^{k+2-n}\ge0
\end{equation}
in Appendix~\ref{app:low-code}.
Then we prove $B_{r,n}(j)\ge0$ for $b\ne n(n+1)/4$.

\emph{Case 2: $b=n(n+1)/4$.}
By the definition of $y$ and the upper bound
$j\le\lfloor\mu_{r,n}\rfloor$ in Proposition~\ref{sc-B}, we have
\begin{align}\label{eq:rb-y}
	r\left(\frac{n(n+1)}4-y\right)
	=\frac{(r-1)n(n+1)}4-j+\frac12
	=\mu_{r,n}-j+\frac12\ge\frac12.
\end{align}
By~\eqref{eq:Ly-bound2} and~\eqref{direct:B-lower-bound}, we have
$$
B_{r,n}(j)\ge
r^{n-2}\left(\frac{n(n+1)}4-y\right)\,E_{[a,n(n+1)/4],\widehat L}
-\sum_{k=0}^{n-3}D_{I,k}^{r'}r^k.
$$
We will prove 
\begin{equation}\label{low:central-cert}
	\frac{E_{[a,n(n+1)/4],\widehat L}}{2}
	-\sum_{k=0}^{n-3}D_{I,k}^{r'}21^{k+3-n}\ge0
\end{equation}
in Appendix~\ref{app:low-code}. 
Since $D_{I,k}\ge 0$ by~\eqref{direct:D-positive},
we have 
$E_{[a,n(n+1)/4],\widehat L}\ge0$, and~\eqref{eq:rb-y} implies
\begin{align*}
	B_{r,n}(j)
	&\ge r^{n-3}\left(\frac{E_{[a,n(n+1)/4],\widehat L}}{2}
	-\sum_{k=0}^{n-3}D_{I,k}^{r'}r^{k+3-n}\right)\\
	&\ge r^{n-3}\left(\frac{E_{[a,n(n+1)/4],\widehat L}}{2}
	-\sum_{k=0}^{n-3}D_{I,k}^{r'}21^{k+3-n}\right)\ge 0,
\end{align*}
where the last inequality uses $k+3-n\le0$.
This completes the proof. \hfill$\square$

\begin{remark}
	The left-hand sides of~\eqref{low:noncentral-cert} and~\eqref{low:central-cert}
	are rational numbers depending only on $n$, $I$, and $r'$, independent of $r$.
	There are $32$ subintervals for even $r\ge22$ with $5\le n\le11$,
	and $24$ subintervals $I$ for odd $r\ge21$ with $n=11$.
	The common bounds already include all possible $j$, so there are
	$56$ inequalities to check.
	The Maple program in Appendix~\ref{app:low-code} implements these checks.
	All $56$ values are positive.
	Consequently, both~\eqref{low:noncentral-cert} and~\eqref{low:central-cert}
	hold on every required subinterval, proving $B_{r,n}(j)\ge0$ in both cases.
\end{remark}

\subsection{Completion of the proof for $n\le 11$}\label{sub-3.4}
We now prove  Theorem~\ref{thm:main} for $1\le n\le 11$.
It suffices to prove the case $r\ge 21$.
\begin{theorem}\label{low:seeds}
	Theorem~\ref{thm:main} holds for $1\le n\le 11$, that is,
	\begin{itemize}
		\item[\rm(i)]  for every even $r\ge 22$, the polynomial $F_{r,n}(q)$ is
		unimodal for $1\le n\le 11.$
		\item[\rm(ii)] for every odd $r\ge 21$ and $n=11$, the polynomial $F_{r,11}(q)$ is
		unimodal.
	\end{itemize}
\end{theorem}
\begin{proof}
	First we prove (i) for even $r\ge 22$.
	The case $n=1$ is trivial.
	For $n=2$, it is known~\cite{AM25} that $F_{r,2}(q)$ is
	unimodal if $r$ is even.
	By Propositions~\ref{sc:three} and~\ref{sc:small}, we prove the unimodality of $F_{r,n}(q)$ for $n=3,4$.
	For $n=5,\ldots,11$, the unimodality of $F_{r,n}(q)$ follows inductively from Theorem~\ref{checkinginterval}
	and~Proposition~\ref{sc-B}~(i).
	
	We now prove (ii) for odd $r\ge 21.$
	By Proposition~\ref{sc-B}~(ii), we have $B_{r,11}(j)\ge0$ for
	$9r\le j\le33(r-1)=\mu_{r,11}$.
	By Proposition~\ref{fd:window}, to prove that $F_{r,11}(q)$
	is unimodal, it suffices to prove the following result.
	
	\begin{claim}\label{cla:br11}
		For every odd integer $r\ge21$, we have
		\begin{equation*}
			B_{r,11}(j)\ge0
			\qquad\text{for}\quad 0\le j<9r.
		\end{equation*}
	\end{claim}
	\noindent\textbf{Proof of Claim~\ref{cla:br11}.}
	Let
	$$
	K_r(q)=\frac{\prod_{k=1}^{8}(1-q^{rk})} {\prod_{k=2}^{11}(1-q^k)}.
	$$
	Then we have
	$$
	(1-q)F_{r,11}(q)=K_r(q)\prod_{k=9}^{11}(1-q^{rk}).
	$$
	For $0\le j<9r$, 
	$$
	B_{r,11}(j)
	=[q^j]\bigl((1-q)F_{r,11}(q)\bigr)
	=[q^j]K_r(q).
	$$
	It suffices to show that $[q^j]K_r(q)\ge0$
	for $0\le j<9r$.
	
	For every $r\ge 21$,
	there exist an $h\in\{9,11\}$ and  a $t\in\{0,1,2\}$ such that $r\equiv ht \pmod{5}$.
	Moreover, there exists an integer $s\ge 0$ such that $r=5s+ht$.
	Indeed, the numbers $0,11,22,18,9$ represent all residue classes modulo $5$.
	
	Let $\ell=20-h$.
	Then $\{h,\ell\}=\{9,11\}.$
	Since $r$ and $\ell$ are odd and $r\ge21>\ell$,
	$(r-\ell)/2$ is a positive integer.
	Rearranging the denominator factors of $K_r(q)$, we obtain
	\begin{align*}
		K_r(q)
		&=	\frac{1-q^r}{(1-q^2)(1-q^\ell)}
		\frac{1-q^{2r}}{(1-q^{10})(1-q^h)}	\prod_{k=3}^{8}\frac{1-q^{rk}}{1-q^k}\\
		&=
		\frac{1-q^r}{(1-q^2)(1-q^\ell)}
		\frac{1-q^{10s+2ht}}{(1-q^{10})(1-q^h)}
		\prod_{k=3}^{8}\left(\sum_{a=0}^{r-1}q^{ka}\right)\\
		&=
		\left(
		\frac{1-q^{r-\ell}}{(1-q^2)(1-q^\ell)}
		+\frac{q^{r-\ell}}{1-q^2}
		\right)
		\left(
		\frac{1-q^{10s}}{(1-q^{10})(1-q^h)}
		+\frac{q^{10s}(1-q^{2ht})}{(1-q^{10})(1-q^h)}
		\right)\\
        &\qquad \qquad\cdot
		\prod_{k=3}^{8}\left(\sum_{a=0}^{r-1}q^{ka}\right)\\
		&=
		\left(
		\frac{\sum\nolimits_{a=0}^{(r-\ell)/2-1}q^{2a}}{1-q^\ell}
		+\frac{q^{r-\ell}}{1-q^2}
		\right)
		\left(
		\frac{\sum\nolimits_{a=0}^{s-1}q^{10a}}{1-q^h}
		+\frac{q^{10s}\sum\nolimits_{a=0}^{2t-1}q^{ha}}{1-q^{10}}
		\right)
		\prod_{k=3}^{8}\left(\sum_{a=0}^{r-1}q^{ka}\right),
	\end{align*}
	where an empty sum is interpreted as zero.
	All coefficients of each power series
	$(1-q^k)^{-1}=\sum_{m\ge0}q^{km}$ are nonnegative. Thus all coefficients of $K_r(q)$ are nonnegative.
	This proves Claim~\ref{cla:br11}.
\end{proof}

\section{Proof of Theorem~\ref{thm:main} for $r\geq21$ and $n\ge 12$}\label{sec:fourier}
In this section, we first derive a Fourier integral representation of $B_{r,n}(j)$
and split the integral into a main term and a tail term.
We then estimate these two terms separately and show that the
main term exceeds the absolute value of the tail.
Using the initial cases established in Section~\ref{sec:lowrows},
we complete the proof of Theorem~\ref{thm:main} by induction on $n$.
Throughout this section, $r\ge21$ and $n\ge12$.

\subsection{Fourier integral representation of $B_{r,n}(j)$}
Recall that
$$
F_{r,n}(q) = \prod_{k=1}^{n} \left(1+q^k+\cdots+q^{(r-1)k}\right) = \sum_{j=0}^{D_{r,n}} A_{r,n}(j)q^j,
$$
where
$D_{r,n}={(r-1)n(n+1)}/{2},$
and $B_{r,n}(j)=A_{r,n}(j)-A_{r,n}(j-1).$
We will derive a Fourier integral representation of $B_{r,n}(j)$,
see~\eqref{eq:inversion}.
First, we evaluate $F_{r,n}(q)$
on the unit circle.
By~\eqref{eq:main}, we obtain 
\begin{align}\label{eq:F-centered}
	F_{r,n}(q)&=\prod_{k=1}^{n} q^{\frac{(r-1)k}{2}} \prod_{k=1}^{n}
	\left(q^{-\frac{(r-1)k}{2}}+q^{-\frac{(r-3)k}{2}}+\cdots+q^{\frac{(r-1)k}{2}}\right)\nonumber\\
	&=\prod_{k=1}^{n} q^{\frac{(r-1)k}{2}} \prod_{k=1}^{n}
	\left(\sum_{\ell=0}^{r-1}q^{k(\ell-(r-1)/2)}\right).
\end{align}
Recall that
$\mu_{r,n}=D_{r,n}/2={(r-1)n(n+1)}/{4}.$
For real $\theta$, substituting $q=e^{i\theta}$ into~\eqref{eq:F-centered} gives
\begin{align*}
	F_{r,n}(e^{i\theta}) &= \exp\left( \frac{i(r-1)\theta}{2}\sum_{k=1}^{n}k \right) \prod_{k=1}^{n} \left(
	\sum_{\ell=0}^{r-1} e^{i(\ell-(r-1)/2)k\theta} \right)\\
	&=e^{i \mu_{r,n}\theta} \prod_{k=1}^{n} \left( \sum_{\ell=0}^{r-1} e^{i(\ell-(r-1)/2)k\theta} \right)\\
	&= r^n e^{i \mu_{r,n}\theta} \prod_{k=1}^{n} \left( \frac{1}{r}\sum_{\ell=0}^{r-1}
	e^{i(\ell-(r-1)/2)k\theta} \right).
\end{align*}
Removing the phase factor $e^{i\mu_{r,n}\theta}$ and
normalizing by $F_{r,n}(1)=r^n$, we define
\begin{equation}\label{def-varphi}
	\varphi_{r,n}(\theta) := \frac{e^{-i\mu_{r,n}\theta}}{r^n} F_{r,n}(e^{i\theta})=
	\prod_{k=1}^{n}\psi_r(k\theta),
\end{equation}
where
\begin{equation}\label{eq:psi}
	\psi_r(x) = \frac{1}{r}\sum_{\ell=0}^{r-1} e^{i(\ell-(r-1)/2)x}.
\end{equation}

The following lemma collects some  basic properties of $\psi_r$ and $\varphi_{r,n}$.
See, for instance,~\cite[Sections~3.3 and~3.5]{Dur19}.
We include the proofs for completeness.
We say a function $f(x)$ is real if $f(x)\in\mathbb{R}$ for $x\in\mathbb{R}$.
A function $f(x)$ is even if $f(-x)=f(x).$

\begin{lemma}\label{realeven}
	We have
	\begin{itemize}
		\item[\rm{(i)}]  $\psi_r(x)$ is real and even. For $x\notin2\pi\mathbb{Z}$, we have $|\psi_r(x)|<1$ and
		\begin{equation}\label{eq:psi-sin}
			\psi_r(x)=\frac{\sin(rx/2)}{r\sin(x/2)}.
		\end{equation}
		For $x=2\pi s$ with $s\in \mathbb{Z}$, we have $\psi_r(x)=(-1)^{(r-1)s}$.
		\item[\rm{(ii)}]  $\varphi_{r,n}(\theta)$ is real and even. For $\theta\notin2\pi\mathbb{Z}$,
		we have $|\varphi_{r,n}(\theta)|<1$. For $\theta\in 2\pi\mathbb{Z}$, we have $|\varphi_{r,n}(\theta)|=1$.
	\end{itemize}
\end{lemma}
\begin{proof}
	We first prove~(i). 
	Pairing the terms in~\eqref{eq:psi} with indices $\ell$ and $r-1-\ell$ gives
	$$
	\psi_r(x) = \frac{1}{r}\sum_{\ell=0}^{r-1} \cos\bigl((\ell-(r-1)/2)x\bigr),
	$$
	so $\psi_r(x)$ is real and even.
	For $x\notin2\pi\mathbb{Z}$ (i.e., $\sin(x/2)\ne 0$), by~\eqref{eq:psi}, we obtain
	$$
	\psi_r(x) = \frac{e^{-i(r-1)x/2}}{r} \frac{1-e^{irx}}{1-e^{ix}}.
	$$
	Using the identity
	$$
	1-e^{iy} = e^{iy/2}\bigl(e^{-iy/2}-e^{iy/2}\bigr) = -2i\,e^{iy/2}\sin(y/2),
	$$
	with $y=rx$ and $y=x$, respectively, we obtain that 
	$$
	\psi_r(x)=\frac{e^{-i(r-1)x/2}}{r} \frac{1-e^{irx}}{1-e^{ix}}
	= \frac{e^{-i(r-1)x/2}}{r}\frac{-2i\,e^{irx/2}\sin(rx/2)} {-2i\,e^{ix/2}\sin(x/2)}
	= \frac{\sin(rx/2)}{r\sin(x/2)}.
	$$
	This proves~\eqref{eq:psi-sin}.
	
	By~\eqref{eq:psi},  we have $|\psi_r(x)|\le 1$.
	Equality holds if and only if all summands have the same
	argument. Since the ratio of consecutive summands is
	$e^{ix}$, this occurs precisely when $x\in2\pi\mathbb{Z}$.
	For $x=2\pi s$, direct substitution gives
	$$
	\psi_r(2\pi s)=\frac{1}{r}\sum_{\ell=0}^{r-1}
	e^{i(\ell-(r-1)/2)\cdot 2\pi s}=e^{-i(r-1)\pi s}=(-1)^{(r-1)s}.
	$$
	
	We now prove~(ii). We have proved that $\psi_{r}(\theta)$ is real and even.
	By~\eqref{def-varphi}, $\varphi_{r,n}(\theta)$ is also
	real and even. If $\theta \notin2\pi\mathbb{Z}$, then
	$$
	|\varphi_{r,n}(\theta)| = \prod_{k=1}^{n}|\psi_r(k\theta)| \le |\psi_r(\theta)|<1.
	$$
	If $\theta\in2\pi\mathbb{Z}$, then every factor $|\psi_r(k\theta)|=1$,
	so $|\varphi_{r,n}(\theta)|=1$.
	This completes the proof.
\end{proof}

Recall that, by Theorem~\ref{checkinginterval}, if $F_{r,n-1}$ is unimodal,
to prove that $F_{r,n}$ is unimodal,
it suffices to verify $B_{r,n}(j)\ge0$ for
$\lfloor\mu_{r,n-1}\rfloor<j\le\lfloor\mu_{r,n}\rfloor.$
Throughout this section, we use the index $t$ to denote the distance
from $j-1/2$ to the center $\mu_{r,n}$, i.e.,
\begin{equation}\label{fd:t}
	t:=\mu_{r,n}-j+\frac{1}{2}.
\end{equation}
Then
\begin{equation}\label{bound:t}
	0<t\le\mu_{r,n}-\mu_{r,n-1}+\frac12=\frac{(r-1)n+1}{2}\le\frac{rn}{2}.
\end{equation}
The following lemma gives a Fourier integral representation of
$B_{r,n}(j)$ and a quadratic lower bound for $\varphi_{r,n}$.

\begin{lemma}\label{lem:inversion}
Let $t=\mu_{r,n}-j+1/2$, as defined in~\eqref{fd:t}.
	For every integer $j$, we have
	\begin{equation}\label{eq:inversion}
		B_{r,n}(j)=\frac{2r^n}{\pi}
		\int_0^\pi\varphi_{r,n}(\theta)
		\sin(t\theta)\sin(\theta/2)\,d\theta.
	\end{equation}
	Moreover, 
	\begin{equation}\label{eq:variance-bound}
		\varphi_{r,n}(\theta)\ge1-\frac{\sigma^2\theta^2}{2},
	\end{equation}
	where 
	\begin{align}\label{eq:sigma}
		\sigma^2=\frac{(r^2-1)n(n+1)(2n+1)}{72}.
	\end{align}
\end{lemma}
\begin{proof}
	We first express $A_{r,n}(j)$ as a Fourier integral and propose the following claim.
	\begin{claim}\label{cla:Arn}
		For every integer $j$, we have 
		\begin{equation}\label{eq:A-inversion}
			A_{r,n}(j)
			=\frac{r^n}{\pi}\int_0^\pi
			\varphi_{r,n}(\theta)
			\cos\bigl((\mu_{r,n}-j)\theta\bigr)\,d\theta.
		\end{equation}
	\end{claim}
	Then by~\eqref{eq:A-inversion}, we obtain the  integral representation of the first difference $B_{r,n}(j)$ as
	\begin{align*}
		B_{r,n}(j)
		&=A_{r,n}(j)-A_{r,n}(j-1)\\
		&=\frac{r^n}{\pi}\int_0^\pi
		\varphi_{r,n}(\theta)
		\Bigl[
		\cos\bigl((\mu_{r,n}-j)\theta\bigr)
		-\cos\bigl((\mu_{r,n}-j+1)\theta\bigr)
		\Bigr]\,d\theta\\
		&=\frac{2r^n}{\pi}\int_0^\pi
		\varphi_{r,n}(\theta)\sin(t\theta)\sin(\theta/2)\,d\theta,
	\end{align*}
	which proves~\eqref{eq:inversion}.
	We now prove Claim~\ref{cla:Arn} and \eqref{eq:variance-bound}.
	
	\noindent\textbf{Proof of Claim~\ref{cla:Arn} and \eqref{eq:variance-bound}.}
	Recall that by~\eqref{eq:main}, 
	$$
	A_{r,n}(j)=\#\left\{(x_1,\ldots,x_n)\in\{0,\ldots,r-1\}^n: \sum_{k=1}^{n}kx_k=j\right\}.
	$$
	Let $X_1,\ldots,X_n$ be independent random variables,
	each uniformly distributed on $\{0,\ldots,r-1\}$,
	and let $Y=\sum_{k=1}^{n}kX_k$. Then
	$$
	\Pr(Y=j)=\frac{A_{r,n}(j)}{r^n}.
	$$
	
	For each $X_k$, the expectation is
	$$
	\E[X_k]=\frac{1}{r}\sum_{\ell=0}^{r-1}\ell
	=\frac{r-1}{2}
	$$
	and the variance satisfies
	\begin{align*}
		\Var(X_k)
		=&\E[X_k^2]-(\E[X_k])^2\\
		=&\frac{1}{r}\sum_{\ell=0}^{r-1}\ell^2
		-\left(\frac{r-1}{2}\right)^2\\
		=&\frac{(r-1)(2r-1)}{6}-\frac{(r-1)^2}{4}\\
		=&\frac{r^2-1}{12}.
	\end{align*}
	Consequently, by independence,
	\begin{align}
		\E[Y]
		&=\sum_{k=1}^{n}k\,\E[X_k]
		=\frac{(r-1)n(n+1)}{4}
		=\mu_{r,n},\notag\\
		\Var(Y)
		&=\sum_{k=1}^{n}k^2\Var(X_k)
		=\frac{(r^2-1)n(n+1)(2n+1)}{72}
		=\sigma^2.\label{var-Y}
	\end{align}
	Moreover,
	\begin{align}\label{eq:E-eiX}
		\mathbb{E}\!\left[e^{i(X_k-(r-1)/2)k\theta}\right]
		&=\sum_{\ell=0}^{r-1}
		e^{i(\ell-(r-1)/2)k\theta}\Pr(X_k=\ell)=\frac{1}{r}\sum_{\ell=0}^{r-1}
		e^{i(\ell-(r-1)/2)k\theta}.
	\end{align}
	Hence, by~\eqref{def-varphi}, \eqref{eq:psi},
	and~\eqref{eq:E-eiX}, we have
	\begin{align}\label{eq:prob}
		\varphi_{r,n}(\theta)
		&=\prod_{k=1}^{n}\psi_r(k\theta)=\prod_{k=1}^{n}
		\left[
		\frac{1}{r}\sum_{\ell=0}^{r-1}
		e^{i(\ell-(r-1)/2)k\theta}
		\right]=\prod_{k=1}^{n}
		\mathbb{E}\!\left[e^{i(X_k-(r-1)/2)k\theta}\right].
	\end{align}
	Since
	$$
	Y-\mu_{r,n}
	=Y-\frac{(r-1)n(n+1)}{4}
	=\sum_{k=1}^{n}k\left(X_k-\frac{r-1}{2}\right),
	$$
	independence and~\eqref{eq:prob} give
	\begin{align}\label{eq:phi-expectation}
		\varphi_{r,n}(\theta)
		&=\prod_{k=1}^{n}
		\mathbb{E}\!\left[e^{i(X_k-(r-1)/2)k\theta}\right]
		=\mathbb{E}\!\left[e^{i\theta(Y-\mu_{r,n})}\right]
		=\mathbb{E}\!\left[e^{i\theta(Y-\E[Y])}\right].
	\end{align}
	
	Since $\varphi_{r,n}$ is real by Lemma~\ref{realeven}~(ii),
	taking real parts on both sides of~\eqref{eq:phi-expectation} gives
	$$
	\varphi_{r,n}(\theta)
	=\mathbb{E}\!\left[
	\cos\bigl(\theta(Y-\mu_{r,n})\bigr)
	\right].
	$$
	Applying $\cos x\ge1-x^2/2$, we obtain, by~\eqref{var-Y}, that
	\begin{align*}
		\varphi_{r,n}(\theta)
		&\ge\mathbb{E}\!\left[
		1-\frac{\theta^2(Y-\mu_{r,n})^2}{2}
		\right]=1-\frac{\theta^2}{2}
		\mathbb{E}\!\left[(Y-\E[Y])^2\right]=1-\frac{\sigma^2\theta^2}{2}.
	\end{align*}
	This proves~\eqref{eq:variance-bound}.

	We now prove Claim~\ref{cla:Arn} by~\eqref{eq:phi-expectation}.
	Since $Y$ is integer-valued and $j$ is an integer, Fourier orthogonality gives
	$$
	\frac{1}{2\pi}\int_{-\pi}^{\pi}
	e^{i\theta(Y-j)}\,d\theta
	=
	\begin{cases}
		1, & Y=j,\\
		0, & Y\ne j.
	\end{cases}
	$$
	Since $Y$ takes only finitely many values, taking expectations
	and interchanging expectation and integration yields
	\begin{align*}
		\frac{A_{r,n}(j)}{r^n}
		=\Pr(Y=j)
		=\mathbb{E}\!\left[
		\frac{1}{2\pi}\int_{-\pi}^{\pi}
		e^{i\theta(Y-j)}\,d\theta
		\right]
		=\frac{1}{2\pi}\int_{-\pi}^{\pi}
		\mathbb{E}\!\left[e^{i\theta(Y-j)}\right]\,d\theta.
	\end{align*}
	Hence, by~\eqref{eq:phi-expectation}, we get
	\begin{align*}
		\frac{A_{r,n}(j)}{r^n}
		&=\frac{1}{2\pi}\int_{-\pi}^{\pi}
		\mathbb{E}\!\left[
		e^{i\theta(Y-\E[Y])}e^{i\theta(\E[Y]-j)}
		\right]\,d\theta\\
		&=\frac{1}{2\pi}\int_{-\pi}^{\pi}
		\varphi_{r,n}(\theta)
		e^{i(\mu_{r,n}-j)\theta}\,d\theta.
	\end{align*}
	By Lemma~\ref{realeven}, $\varphi_{r,n}$ is real and even.
	The sine term therefore integrates to zero, and the cosine term
	is even. Thus
	$$
	A_{r,n}(j)
	=\frac{r^n}{\pi}\int_0^\pi
	\varphi_{r,n}(\theta)
	\cos\bigl((\mu_{r,n}-j)\theta\bigr)\,d\theta.
	$$
	This proves Claim~\ref{cla:Arn} and completes the proof.
\end{proof}

Recall that, by~\eqref{fd:t} and~\eqref{bound:t}, 
$t=\mu_{r,n}-j+\frac12$ and  $0<t\le rn/2$.
We split the integral interval $[0,\pi]$ at $2\pi/(rn)$ in~\eqref{eq:inversion} as
\begin{align}
	\frac{\pi}{2r^n}	B_{r,n}(j)&=\int_0^\pi\varphi_{r,n}(\theta)\sin(t\theta)\sin(\theta/2)\,d\theta\notag\\
	&=\int_0^{2\pi/(rn)}\varphi_{r,n}(\theta)\sin(t\theta)\sin(\theta/2)\,d\theta+\int_{2\pi/(rn)}^\pi\varphi_{r,n}(\theta)\sin(t\theta)\sin(\theta/2)\,d\theta.\label{eq:main-tail-split}
\end{align}
We call the first integral on the right-hand side of~\eqref{eq:main-tail-split}
the main term and call the second the tail term. 
Write
\begin{align}\label{eq:PJ}
	P_{r,n}(t)=\int_0^{2\pi/(rn)}\varphi_{r,n}(\theta) \sin(t\theta)\sin(\theta/2)\,d\theta\quad
	\textrm{and}\quad J_{r,n}=\int_{2\pi/(rn)}^\pi\theta^2|\varphi_{r,n}(\theta)|\,d\theta.
\end{align}
Since $|\sin x|\le |x|$ for every real $x$, for $t>0,$ we have
$$
|\sin(t\theta)\sin(\theta/2)| \le t\theta\cdot\frac{\theta}{2} = \frac{t\theta^2}{2}.
$$
Hence, the absolute value of the tail integral satisfies
\begin{align*}
	\left| \int_{2\pi/(rn)}^\pi \varphi_{r,n}(\theta) \sin(t\theta)\sin(\theta/2)\,d\theta \right|
	&\le \int_{2\pi/(rn)}^\pi |\varphi_{r,n}(\theta)| |\sin(t\theta)\sin(\theta/2)|\,d\theta\\
	&\le \frac{t}{2} \int_{2\pi/(rn)}^\pi \theta^2|\varphi_{r,n}(\theta)|\,d\theta= \frac{t}{2}J_{r,n}.
\end{align*}
Therefore,
\begin{align}\label{eq:local-tail}
	B_{r,n}(j) \ge \frac{2r^n}{\pi} \left( P_{r,n}(t)-\frac{t}{2}J_{r,n} \right).
\end{align}
In the following, we will show that
$P_{r,n}(t)>\frac{t}{2}J_{r,n}$.

\subsection{A lower bound for the main term $P_{r,n}(t)$}
\label{subsec:main-term}
In this subsection, we will establish a lower bound for
$P_{r,n}(t)$.

\begin{proposition}\label{lem:positive-lobe}
	Let $r\ge21$ and $n\ge12$ be integers.
	For $0<t\le rn/2$, we have
	\begin{equation}\label{315}
		P_{r,n}(t)>\frac{315\,t}{16\,r^3n^{9/2}}.
	\end{equation}
\end{proposition}
\begin{proof}
	Recall that, by \eqref{eq:PJ},
	$$
	P_{r,n}(t)=\int_0^{2\pi/(rn)}\varphi_{r,n}(\theta) \sin(t\theta)\sin(\theta/2)\,d\theta.
	$$
    We first prove the nonnegativity of each factor in the integrand.
	\begin{claim}\label{nonnegative-4.3}
		The functions $\varphi_{r,n}(\theta), \sin(t\theta), \sin(\theta/2)$ are nonnegative for
		$0\le \theta\le 2\pi/(rn).$
	\end{claim}
	\noindent\textbf{Proof of Claim~\ref{nonnegative-4.3}.}  At $\theta=0$, we have $\varphi_{r,n}(0)=1$,
	and $\sin(t\theta)$ and $\sin(\theta/2)$ both vanish.
	For $0<\theta\le 2\pi/(rn)$ and $1\le k\le n$,
	by~\eqref{def-varphi} and~\eqref{eq:psi-sin}, we have
	$$
	\varphi_{r,n}(\theta)= \prod_{k=1}^{n}\psi_r(k\theta)\quad \textrm{and} \quad \psi_r(k\theta)
	=\frac{\sin(rk\theta/2)}{r\sin(k\theta/2)}.
	$$
	Since
	$$
	0<\frac{k\theta}{2}\le\frac{\pi}{r}<\pi \quad \textrm{and} \quad 0<\frac{rk\theta}{2}\le\pi,
	$$
	we have
	$\psi_r(k\theta) \ge0,$
	and hence
	$\varphi_{r,n}(\theta) \ge0.$
	Moreover, since $0<t\le rn/2$, we have 
	$$
	0\le t\theta\le\pi \quad \textrm{and} \quad 0\le\theta/2\le\pi/(rn)<\pi.
	$$
	Thus $\sin(t\theta)$ and $\sin(\theta/2)$ are also
	nonnegative.
	This proves Claim~\ref{nonnegative-4.3}.
	
	Recall that, by~\eqref{eq:sigma},
	$$
	\sigma^2=\frac{(r^2-1)n(n+1)(2n+1)}{72}.
	$$
	Since $r\ge21$ and $n\ge12$, we have $r^2-1\ge3r^2/4$ and hence
	\begin{equation}\label{sigmarn}
		\sigma^2
		\ge\frac{3r^2}{4}\cdot\frac{2n^3}{72}
		=\frac{r^2n^3}{48}
		\ge\frac{r^2n^2}{4}.
	\end{equation}
	Therefore, $\sigma\ge rn/2\ge126$ and
	$1/\sigma\le2/(rn)<2\pi/(rn)$.
	Since $\varphi_{r,n}(\theta)$, $\sin(t\theta)$,
	and $\sin(\theta/2)$ are nonnegative on $[0,2\pi/(rn)]$,
	we obtain
	\begin{equation}\label{eq:est-p}
		P_{r,n}(t)\ge
		\int_0^{1/\sigma}\varphi_{r,n}(\theta)
		\sin(t\theta)\sin(\theta/2)\,d\theta.
	\end{equation}
	
	We will use~\eqref{eq:est-p} to prove~\eqref{315}.
	Before that, we obtain the following bounds.
	\begin{claim}\label{nonnegative-4.3-1}
		For $r\ge21$ and $n\ge12$, we have \begin{align} \label{eq:sigma3}
			\frac{t^2}{\sigma^2}<\frac23 \qquad \textrm{and} \qquad
			\sigma^3 < \frac{r^3n^{9/2}}{180}.   
		\end{align}
	\end{claim}
	\noindent\textbf{Proof of Claim~\ref{nonnegative-4.3-1}.}
	Since $t\le rn/2$ by~\eqref{bound:t}, we have
	$$
	\frac{t^2}{\sigma^2}\le \frac{r^2n^2}{4} \frac{72}{(r^2-1)n(n+1)(2n+1)}
	\le\frac{18r^2n}{(r^2-1)(n+1)(2n+1)}.
	$$
	The expression on the right is decreasing in each of $r$ and $n$.
	A direct calculation yields
	for $r\ge21$ and $n\ge12$,
	\begin{equation*}
		\frac{t^2}{\sigma^2}\le \frac{18\cdot21^2\cdot12}{(21^2-1)\cdot13\cdot25}<\frac23.
	\end{equation*}
	This proves the first inequality of~\eqref{eq:sigma3}. 
	For the second inequality, by definition~\eqref{eq:sigma}, since $n\ge 12,$
	\begin{align*}
		\sigma^2
		&=\frac{(r^2-1)n(n+1)(2n+1)}{72}=\frac{(r^2-1)n^3}{72}
		\left(2+\frac{3}{n}+\frac{1}{n^2}\right)\\
		&\le\frac{(r^2-1)n^3}{72}
		\left(2+\frac14+\frac{1}{144}\right)=\frac{325}{10368}(r^2-1)n^3
		<\frac{325}{10368}r^2n^3.
	\end{align*}
	It follows that
	\begin{align*}
		\sigma^3 < \left(\frac{325}{10368}\right)^{3/2}r^3n^{9/2} < \frac{r^3n^{9/2}}{180}.
	\end{align*}
	This completes the proof of Claim~\ref{nonnegative-4.3-1}.

	Now we are ready to give an estimate of $P_{r,n}(t)$ on $\theta\in[0,1/\sigma]$.
	Applying~\eqref{eq:variance-bound} and
	$\sin x\ge x-x^3/6$ with $0\le \theta\le 1/\sigma$ yields
	\begin{align}\label{threeinequality}
		&\varphi_{r,n}(\theta)\ge1-\frac{\sigma^2\theta^2}{2},
		\quad \sin(t\theta)\ge t\theta \left(1-\frac{t^2\theta^2}{6}\right),
		\quad
		\sin(\theta/2)\ge \frac{\theta}{2} \left(1-\frac{\theta^2}{24}\right).
	\end{align}
	Hence, combining~\eqref{eq:est-p} with~\eqref{threeinequality} gives
	\begin{align}\label{low-P-1}
		P_{r,n}(t)&\ge \int_0^{1/\sigma}\varphi_{r,n}(\theta) \sin(t\theta)\sin(\theta/2)\,d\theta\notag\\
		&\ge \frac{t}{2} \int_0^{1/\sigma} \theta^2 \left(1-\frac{\sigma^2\theta^2}{2}\right)
		\left(1-\frac{t^2\theta^2}{6}\right) \left(1-\frac{\theta^2}{24}\right) \,d\theta\notag\\
		&\ge \frac{t}{2} \int_0^{1/\sigma} \theta^2 \left(1-\frac{\sigma^2\theta^2}{2}\right) \left[
		1-\left(\frac{t^2}{6}+\frac{1}{24}\right)\theta^2 \right] \,d\theta.
	\end{align}
	Substituting $y=\sigma\theta$ in~\eqref{low-P-1} yields
	\begin{align}\label{PRN}
		P_{r,n}(t) &\ge \frac{t}{2\sigma^3} \int_0^1 y^2\left(1-\frac{y^2}{2}\right) \left[ 1-\left(
		\frac{t^2}{6\sigma^2}+\frac{1}{24\sigma^2} \right)y^2 \right] \,dy\nonumber\\
		&= \frac{t}{2\sigma^3} \left( \frac{7}{30} -\frac{3t^2}{140\sigma^2} -\frac{3}{560\sigma^2} \right)>\frac{7t}{64\sigma^3},
	\end{align}
	where the last inequality follows from
	$t^2/\sigma^2<2/3$ (by the first inequality in~\eqref{eq:sigma3}) and $\sigma^2>r^2n^2/4\ge 126^2$ (by~\eqref{sigmarn} with $r\ge 21,n\ge 12$).  
	
	Combining \eqref{PRN} with the second inequality in~\eqref{eq:sigma3}, we obtain
	$$
	P_{r,n}(t) > \frac{7t}{64\sigma^3} > \frac{315}{16}\frac{t}{r^3n^{9/2}}.
	$$
	This completes the proof.
\end{proof}

\subsection{An upper bound for the tail term $J_{r,n}$}
Recall that, by~\eqref{eq:PJ},
$$
J_{r,n}=\int_{2\pi/(rn)}^\pi\theta^2|\varphi_{r,n}(\theta)|\,d\theta.
$$
We split the integral
$J_{r,n}=J_{r,n}^{(1)}+J_{r,n}^{(2)}+J_{r,n}^{(3)}$ as
\begin{align}
	J_{r,n}^{(1)}&=\int_{2\pi/(rn)}^{\pi/n}\theta^2|\varphi_{r,n}(\theta)|\,d\theta,\label{3j1}\\
	J_{r,n}^{(2)}&=\int_{\pi/n}^{\pi/10}\theta^2|\varphi_{r,n}(\theta)|\,d\theta,\label{3j2}\\
	J_{r,n}^{(3)}&=\int_{\pi/10}^{\pi}\theta^2|\varphi_{r,n}(\theta)|\,d\theta. \label{3j3}
\end{align}
We will give upper bounds for $J_{r,n}^{(1)}$,
$J_{r,n}^{(2)}$, and $J_{r,n}^{(3)}$ separately.

Before estimating these integrals, we will record an elementary bound for repeated use below.
By the concavity of sine on $[0,\pi/2]$,
Jensen's inequality
(see \cite[Theorem~1.6.2]{Dur19}) gives
\begin{equation}\label{eq:sine-chord}
	\sin x\ge\frac{\sin a}{a}\,x \qquad \text{for}\quad 0\le x\le a,\,\, 0<a\le\pi/2.
\end{equation}

\subsubsection{An upper bound for the tail term $J^{(1)}_{r,n}$}
\begin{proposition}\label{lem:uniform-J0}
	Let $r\ge21$ and $n\ge12$ be integers. Then 
	\begin{align}\label{eq:J1rn}
		J^{(1)}_{r,n}<\frac{63}{2r^3n^{9/2}}.
	\end{align}
\end{proposition}

\begin{proof}
	Let $z=r\theta/2$.
	Since $\pi<8<r\pi/2$,  by~\eqref{3j1}, we obtain that
	\begin{align}\label{eq:Jrn1}
		r^3J^{(1)}_{r,n}
		&=r^3\int_{2\pi/(rn)}^{\pi/n}
		\theta^2|\varphi_{r,n}(\theta)|\,d\theta\notag\\
		&=8\int_{\pi/n}^{r\pi/(2n)}
		z^2|\varphi_{r,n}(2z/r)|\,dz\notag\\
		&=8\int_{\pi/n}^{8/n}
		z^2|\varphi_{r,n}(2z/r)|\,dz
		+8\int_{8/n}^{r\pi/(2n)}
		z^2|\varphi_{r,n}(2z/r)|\,dz.
	\end{align}	
	To obtain an upper bound for $J^{(1)}_{r,n}$, we give the following two claims.
	\begin{claim}\label{cla:Jup-bound1}
		We have 
		\begin{align}
			&8\int_{\pi/n}^{8/n} z^2|\varphi_{r,n}(2z/r)|\,dz \le \frac{8}{(n-3)n!} \left(\frac{91}{90}\right)^n
			\left(\frac n\pi\right)^{n-3}.\label{claim 6-1}\\
			&8\int_{8/n}^{r\pi/(2n)} z^2|\varphi_{r,n}(2z/r)|\,dz 
			\le \frac8{(n-3)n!} \left(\frac{11}{7}\right)^n \left(\frac n8\right)^{n-3}.\label{claim 6-2}
		\end{align}
	\end{claim}
	
	By~\eqref{eq:Jrn1},~\eqref{claim 6-1} and~\eqref{claim 6-2}, using $31/10<\pi<22/7$, we obtain
	\begin{equation}\label{Zn}
		r^3 J^{(1)}_{r,n}\le \frac{8}{(n-3)n!} \left[ \left(\frac{91}{90}\right)^n
		\left(\frac{10n}{31}\right)^{n-3} + \left(\frac{11}{7}\right)^n \left(\frac n8\right)^{n-3} \right].
	\end{equation}
	\begin{claim}\label{cla:nZn}
		Let $Z_n$ denote the right-hand side of~\eqref{Zn}.
		Then $n^{9/2}Z_n$ decreases with $n$ for integers $n\ge12$.
	\end{claim}
	By Claim~\ref{cla:nZn}, we have 
	\begin{align*}
		r^3 n^{9/2} J^{(1)}_{r,n}&\le\frac{8 n^{9/2}}{(n-3)n!} \left[ \left(\frac{91}{90}\right)^n
		\left(\frac{10n}{31}\right)^{n-3} + \left(\frac{11}{7}\right)^n \left(\frac n8\right)^{n-3} \right]\\
		&\le \frac{8\cdot 12^{9/2}}{9\cdot 12!} \left[ \left(\frac{91}{90}\right)^{12}
		\left(\frac{120}{31}\right)^9 + \left(\frac{11}{7}\right)^{12} \left(\frac32\right)^9 \right]
		<\frac{63}{2}.
	\end{align*}
	This proves \eqref{eq:J1rn}. 
	We now give the proofs of Claims~\ref{cla:Jup-bound1} and~\ref{cla:nZn}.
	
	\noindent\textbf{Proof of~\eqref{claim 6-1} in Claim~\ref{cla:Jup-bound1}.}
	We first estimate $|\varphi(2z/r)|$. Clearly, $0<2z/r<\pi$ by $z\le 8/n$ and $r\ge 21, n\ge 12$.
	Recall that by definition~\eqref{def-varphi} of $\varphi$ and expression~\eqref{eq:psi-sin},
	we have 
	\begin{align}\label{var-1}
		|\varphi_{r,n}(2z/r)| = \prod_{k=1}^{n} \frac{|\sin(kz)|}{r\sin(kz/r)}\le \prod_{k=1}^{n}
		\frac{1}{r\sin(kz/r)}.
	\end{align}
	For $0\le u\le1$, the Taylor estimate gives ${\sin u}/{u}\ge1-{u^2}/{6},$
	where the quotient at $u=0$ is interpreted by continuity.
	Moreover,
	$$
	-\log\left(1-\frac{u^2}{6}\right) = \int_0^{u^2/6}\frac{dt}{1-t}
	\le\frac{1}{1-u^2/6}\cdot \frac{u^2}{6}= \frac{u^2}{6-u^2} \le\frac{u^2}{5}.
	$$
	Consequently, for $0\le u\le1$, we have
	\begin{equation}\label{eq:sinc-lower}
		\frac{\sin u}{u}\ge1-\frac{u^2}{6}\ge e^{-u^2/5}.
	\end{equation}
	Since  $0<{kz}/{r} \le n\cdot(8/n)/r\le 8/{21}<1,$
	applying~\eqref{eq:sinc-lower} to $u=kz/r$, we obtain
	$$
	r\sin(kz/r)\ge
	kz\exp\left(-\frac{k^2z^2}{5r^2}\right).
	$$
	Thus by~\eqref{var-1},
	\begin{align}\label{var-2}
		|\varphi_{r,n}(2z/r)| \le \prod_{k=1}^{n}
		\frac{1}{r\sin(kz/r)} \le \prod_{k=1}^{n} \frac{1}{kz}\exp\left(\frac{k^2z^2}{5r^2}\right)=
		\frac{\exp\left({S_nz^2}/{5r^2}\right)}{n!z^n},
	\end{align}
	where
	$$
	S_n=\sum_{k=1}^{n}k^2=\frac{n(n+1)(2n+1)}6.
	$$
	For $z\le 8/n,$ $n\ge 12$, since $S_n/n^3=(1+1/n)(2+1/n)/6$ is decreasing with $n$, we have 
	\begin{equation}\label{eq:exponent-bound}
		0\le\frac {{S_nz^2}/{5r^2}}{n}\le \frac{64}{n^2}\frac{S_n}{5r^2n}
		\le\frac{64}{5\cdot21^2}\frac{S_n}{n^3}
		\le\frac{64}{5\cdot21^2}\frac{(1+1/12)(2+1/12)}{6}<\frac1{91}.
	\end{equation}
	For $0\le v<1$, we have
	\begin{equation}\label{eq:exp-upper}
		e^v\le\frac{1}{1-v}.
	\end{equation}
	By~\eqref{eq:exponent-bound} and~\eqref{eq:exp-upper}, we obtain
	\begin{equation}\label{9190}
\exp\left(\frac{S_nz^2}{(5r^2)}\right)=\left(\exp\left(\frac{S_nz^2}{(5r^2n)}\right)\right)^n
		\le\left(\frac1{1-{S_nz^2}/{(5r^2n)}}\right)^n <\left(\frac{91}{90}\right)^n.
	\end{equation}
	It follows by~\eqref{var-2} and~\eqref{9190} that
	\begin{align*}
		8\int_{\pi/n}^{8/n} z^2|\varphi_{r,n}(2z/r)|\,dz 
		&\le8\cdot\left(\frac{91}{90}\right)^n
		\int_{\pi/n}^{8/n} \frac{z^2}{n!z^n}\,dz \\
		&\le \frac{8}{(n-3)n!} \left(\frac{91}{90}\right)^n
		\left(\frac n\pi\right)^{n-3}.
	\end{align*}
	This proves~\eqref{claim 6-1}.
	
	\noindent\textbf{Proof of~\eqref{claim 6-2} in Claim~\ref{cla:Jup-bound1}.}
	For $8/n\le z\le r\pi/(2n)$,
	we have
	$0<{kz}/{r}\le {\pi}/{2}$
	for	$1\le k\le n$.
	Taking $a=\pi/2$ in~\eqref{eq:sine-chord},
	we have $\sin x\ge {2x}/{\pi}$ and
	$r\sin(kz/r)\ge  {2kz}/{\pi}.$
	Therefore by~\eqref{var-1},
	\begin{equation}\label{nece}
	|\varphi_{r,n}(2z/r)| \le \prod_{k=1}^{n}
	\frac{1}{r\sin(kz/r)}\le \prod_{k=1}^{n}\frac{\pi}{2kz} = \frac{(\pi/2)^n}{n!z^n}.
	\end{equation}
	Integrating on both sides of~\eqref{nece} gives
	$$
	\begin{aligned}
		8\int_{8/n}^{r\pi/(2n)} z^2|\varphi_{r,n}(2z/r)|\,dz &\le \frac{8(\pi/2)^n}{n!}
		\int_{8/n}^{\infty}z^{2-n}\,dz\\
		&= \frac8{(n-3)n!} \left(\frac{\pi}{2}\right)^n \left(\frac n8\right)^{n-3}\\
		&\le \frac8{(n-3)n!} \left(\frac{11}{7}\right)^n \left(\frac n8\right)^{n-3}.
	\end{aligned}
	$$
	This proves~\eqref{claim 6-2} and completes the proof of Claim~\ref{cla:Jup-bound1}.
	
	\noindent\textbf{Proof of Claim~\ref{cla:nZn}.}
	Recall that 
	$$
	Z_n=\frac{8}{(n-3)n!} \left[ \left(\frac{91}{90}\right)^n
	\left(\frac{10n}{31}\right)^{n-3} + \left(\frac{11}{7}\right)^n \left(\frac n8\right)^{n-3} \right].
	$$
	For $a>0$, define
	\begin{equation}\label{def:T}
		T_n(a)=\frac{a^n n^{n+3/2}}{(n-3)n!}.
	\end{equation}
	Then 
	\begin{align*}
		n^{9/2}Z_n&=\frac{8}{(n-3)n!}\left[ \left(\frac{91}{90}\right)^n \left(\frac{10}{31}\right)^{n-3}
		n^{n+3/2}
		+\left(\frac{11}{7}\right)^n \left(\frac 18\right)^{n-3} n^{n+3/2}
		\right]\\
		&=8\left[
		\left(\frac{31}{10}\right)^3T_n\!\left(\frac{91}{279}\right)
		+8^3T_n\!\left(\frac{11}{56}\right)
		\right].
	\end{align*}
	We now prove that $T_n(a)$ decreases with $n$
	for $a=91/279$ and $a=11/56$.
	By~\eqref{def:T},
	\begin{equation}\label{Tratio}
		\frac{T_{n+1}(a)}{T_n(a)}
		=a\,\frac{n-3}{n-2}
		\left(1+\frac1n\right)^{n+3/2}.
	\end{equation}
	Using
	$$
	\log(1+x)\le x-\frac{x^2}{2(1+x)}
	\quad\text{for }x\ge0\quad \textrm{and} \quad
	\log(1-x)\le-x
	\quad\text{for }0\le x<1,
	$$
	we obtain
	\begin{align*}
		\log\left[
		\frac{n-3}{n-2}
		\left(1+\frac1n\right)^{n+3/2}
		\right]
		&=\log\left(1-\frac1{n-2}\right)
		+\left(n+\frac32\right)\log\left(1+\frac1n\right)\\
		&\le-\frac1{n-2}
		+\left(n+\frac32\right)
		\left(\frac1n-\frac1{2n(n+1)}\right)\\
		&=1-\frac{9n+6}{4n(n+1)(n-2)}\\
		&<1.
	\end{align*}
	Thus by~\eqref{Tratio},
	$$
	\frac{T_{n+1}(a)}{T_n(a)}<ae<3a.
	$$
	Since $3\cdot91/279<1$ and $3\cdot11/56<1$,
	both $T_n(91/279)$ and $T_n(11/56)$ decrease with $n$,
	and so does $n^{9/2}Z_n$.
	This proves Claim~\ref{cla:nZn} and completes the proof.
\end{proof}

\subsubsection{An upper bound for the tail term $J^{(2)}_{r,n}$}
The following lemma is due to Almkvist~\cite{Alm89}.
\begin{lemma}\label{lem:power-envelope}
	For every integer $r\ge21$ and all real $x$ and $\theta$,
	we have
	\begin{align}
		|\psi_r(x)|&\le r^{-\sin^2(x/2)},\label{eq:psi-lem}\\
		|\varphi_{r,n}(\theta)|&\le r^{-\sum_{k=1}^{n}\sin^2(k\theta/2)}.\label{eq:varphi-lem}
	\end{align}
\end{lemma}
\begin{proof}
Equation~\eqref{eq:psi-lem} follows from
\cite[Lemma~2(C)]{Alm89}.
By~\eqref{def-varphi}, we have $\varphi_{r,n}(\theta)= \prod_{k=1}^{n}\psi_r(k\theta)$.
Applying~\eqref{eq:psi-lem} to each factor and taking the product yields~\eqref{eq:varphi-lem}.
\end{proof}

By Lemma~\ref{lem:power-envelope}, an upper bound for $|\varphi_{r,n}(\theta)|$ follows
from a lower bound for $\sum_{k=1}^{n}\sin^2(k\theta/2)$.
We now give a bound for $\sum_{k=1}^{n}\sin^2(k\theta/2)$.
\begin{lemma}\label{lem:sine-square}
	If $n\ge3$ and $\pi/n\le\theta\le\pi$, then
	$$
	\sum_{k=1}^{n}\sin^2(k\theta/2)\ge2n/5.
	$$
\end{lemma}
\begin{proof}
	By the trigonometric summation identity
	\cite[Eq.~(2.11)]{DJ23},
	\begin{equation}\label{DJ}
		\sum_{k=1}^{n}\sin^2(k\theta/2) =\frac n2+\frac14 -\frac{\sin((2n+1)\theta/2)}{4\sin(\theta/2)}.
	\end{equation}
	For $\pi/n\le\theta\le2\pi/n$,  we have
	$\sin(n\theta/2)\ge0$ and
	$\cos((n+1)\theta/2)\le0$, and
	\begin{align*}
		\sum_{k=1}^{n}\sin^2(k\theta/2)&=\frac n2+ \frac{\sin(\theta/2)-\sin((2n+1)\theta/2)}{4\sin(\theta/2)}\\
		&=\frac n2 -\frac{\sin((2n+1)\theta/2)-\sin(\theta/2)}{4\sin(\theta/2)}\\
		&=\frac n2- \frac{\sin(n\theta/2)\cos((n+1)\theta/2)} {2\sin(\theta/2)} \ge\frac n2>\frac{2n}{5}.
	\end{align*}
	For $2\pi/n\le\theta\le\pi$,
	we have $\sin (\theta/2)\ge \sin(\pi/n)$.
	Taking $a=\pi/3$ in~\eqref{eq:sine-chord},
	we have
	$$
	\sin x\ge \frac{3\sqrt3}{2\pi}x \quad \textrm{and} \quad 0\le x\le \pi/3.
	$$
	Thus,
	$$
	\sin(\theta/2)\ge\sin(\pi/n) \ge\frac{3\sqrt3}{2n}>\frac{5}{2n}.
	$$
	Consequently, by~\eqref{DJ},
	$$
	\sum_{k=1}^{n}\sin^2(k\theta/2) \ge\frac n2+\frac14-\frac{1}{4\sin(\theta/2)}
	>\frac{n}{2}+\frac14-\frac{n}{10}>\frac{2n}{5}.
	$$
	This completes the proof.
\end{proof}

We now  establish an upper bound for the tail integral
$J^{(2)}_{r,n}$.

\begin{proposition}\label{lem:large-n-tail}
	For all integers $r\ge 21$ and $n\ge 12$,
	\begin{equation*}
		J^{(2)}_{r,n}<\frac{3}{2r^3n^{9/2}}.
	\end{equation*}
\end{proposition}
\begin{proof}
	Recall that by~\eqref{3j2},
	$$J_{r,n}^{(2)}=\int_{\pi/n}^{\pi/10}\theta^2|\varphi_{r,n}(\theta)|\,d\theta.$$
	For $\frac{\pi}{n}\le\theta\le\pi/10,$ by Lemmas~\ref{lem:power-envelope} and~\ref{lem:sine-square},
	$$
	|\varphi_{r,n}(\theta)|\le r^{-\sum_{k=1}^{n}\sin^2(k\theta/2)}\quad \textrm{and} \quad \sum_{k=1}^{n}\sin^2(k\theta/2)\ge\frac{2n}{5}.
	$$
	Therefore for ${\pi}/{n}\le\theta\le\pi$
	$$
	|\varphi_{r,n}(\theta)|\le r^{-2n/5}.
	$$
	Since $3-2n/5<0$, it follows that
	\begin{align*}
		r^3n^{9/2}J^{(2)}_{r,n}&= r^3n^{9/2}\int_{\pi/n}^{\pi/10}\theta^2|\varphi_{r,n}(\theta)|\,d\theta
		\le n^{9/2}r^{3-2n/5} \int_{\pi/n}^{\pi/10}\theta^2\,d\theta\\
		&=n^{9/2}r^{3-2n/5} \frac{1}{3}\left( \frac{\pi^3}{1000}-\frac{\pi^3}{n^3}\right) \le \frac{\pi^3}{3}n^{9/2} \left(\frac1{1000}-\frac1{n^3}\right) 21^{3-2n/5}.
	\end{align*}
    Let $$
   M(n):=\frac{\pi^3}{3}n^{9/2} \left(\frac1{1000}-\frac1{n^3}\right) 21^{3-2n/5}.
    $$
	Regarding $n\ge12$ as a real variable, we have
	$$
	\frac{d}{dn}\log M(n) = \frac{3}{2n}+\frac{3n^2}{n^3-1000} -\frac25\log21.
	$$
	The first two terms decrease with $n$; in particular,
	$$
	\frac{d}{dn}\log M(n)\le  \frac{3}{2\cdot 12}+\frac{3\cdot 12^2}{12^3-1000} -\frac25\log21 \le \frac{523}{728}-\frac25\log21<0.
	$$
	Thus $M(n)\le M(12)$.
	Using $\pi<22/7$, $\sqrt{12}<7/2$, and $21^{1/5}<2$,
	we obtain
	$$
	M(12) < \frac{(22/7)^3}{3}\cdot 12^4\cdot\frac72 \left(\frac1{1000}-\frac1{12^3}\right) \frac{2}{21^2}
	<\frac32.
	$$
	This completes the proof.
\end{proof}

\subsubsection{An upper bound for the tail term $J^{(3)}_{r,n}$}

In this subsection, we establish a uniform upper bound for $J^{(3)}_{r,n}$
for all integers $r\ge21$.
Recall that by~\eqref{3j3},
$$J_{r,n}^{(3)}=\int_{\pi/10}^{\pi}\theta^2|\varphi_{r,n}(\theta)|\,d\theta.$$
We first consider $n=12$ and derive an upper bound for $r^3 12^{9/2}J^{(3)}_{r,12}.$
\begin{lemma}\label{lem:fixed-twelve}
	For every integer $r\ge21$, we have
	\begin{align}\label{eq:r12-4}
		r^3\,12^{9/2}\int_{\pi/10}^{\pi}
		\theta^2|\varphi_{r,12}(\theta)|\,d\theta<4.
	\end{align}
\end{lemma}

\begin{proof}
	Fix $\theta\in[\pi/10,\pi]$.
	Recall that by~\eqref{def-varphi},
	$\varphi_{r,12}(\theta)=\prod_{k=1}^{12}\psi_r(k\theta).$
	By Lemma~\ref{realeven}, $|\psi_r(x)|\le 1$ and for $k\theta\notin 2\pi\mathbb{Z},$
	$$|\psi_r(k\theta)|=\frac{|\sin(rk\theta/2)|}{|r\sin(k\theta/2)|}\le \frac{1}{r\cdot |\sin(k\theta/2)|}.$$
	Thus, for each $1\le k\le 12$, we have 
	\begin{equation}\label{pvminimum}
		|\psi_r(k\theta)|\le \min\left\{1,\frac1{r\cdot |\sin(k\theta/2)|}\right\},
	\end{equation}
	where the minimum is $1$
	when $k\theta\in2\pi\mathbb{Z}$.

	\begin{claim}\label{cla:convert}
		We have 
		$$r^3\int_{\pi/10}^{\pi}\theta^2|\varphi_{r,12}(\theta)|\,d\theta
		\le 21^3\int_{\pi/10}^{\pi}\theta^2\prod_{k=1}^{12}\min\left\{1,\frac1{21\cdot |\sin(k\theta/2)|}\right\}\,d\theta.
		$$
	\end{claim}
	\noindent\textbf{Proof of Claim~\ref{cla:convert}.}
	We now estimate $|\sin(k\theta/2)|$ for $k=1,2,\ldots,12.$
	A direct computation gives for $\ell=1,2,\ldots, 6,$
	\begin{align}\label{theta2}
		\sin\left({\theta}/{2}\right)&=\sin\left({2\ell\theta}/{2}-{(2\ell-1)\theta}/{2}\right)\nonumber\\
		&=\sin\left({2\ell\theta}/{2}\right)\cos\left({(2\ell-1)\theta}/{2}\right)-\cos\left({2\ell\theta}/{2}\right)\sin\left({(2\ell-1)\theta}/{2}\right)\nonumber\\
		&\le |\sin((2\ell-1)\theta/2)|+|\sin(2\ell\theta/2)|.
	\end{align}
	Taking $a=\pi/2$ in~\eqref{eq:sine-chord},
	we have $\sin x\ge {2x}/{\pi}$.
	Thus, $\sin\left({\theta}/{2}\right)\ge \theta/\pi\ge 1/10$ since $\theta\ge \pi/10.$
	Combining with~\eqref{theta2}, for $\ell=1,2,\ldots,6$, we have 
	$$
	|\sin((2\ell-1)\theta/2)|+|\sin(2\ell\theta/2)|\ge \frac{1}{10}>\frac{2}{21}
	$$
	Thus each pair $(2\ell-1, 2\ell)$ contains at least one index satisfying 
	$|\sin(k\theta/2)|>1/21.$
	Choose one such index from each pair. For each of these six selected
	indices, by~\eqref{pvminimum},
	\begin{equation}\label{psi-1case}
		|\psi_r(k\theta)|\le \frac1{r\cdot |\sin(k\theta/2)|}=\frac{21}{r}\frac1{21\cdot |\sin(k\theta/2)|}=\frac{21}{r}\min\left\{1,\frac1{21\cdot |\sin(k\theta/2)|}\right\}
	\end{equation}
	For the remaining six indices, by~\eqref{pvminimum} with $r\ge 21$, we have 
	\begin{equation}\label{psi-2case}
		|\psi_r(k\theta)|\le \min\left\{1,\frac1{r\cdot |\sin(k\theta/2)|}\right\}\le \min\left\{1,\frac1{21\cdot |\sin(k\theta/2)|}\right\}.
	\end{equation}
	
	Multiplying all twelve factors by~\eqref{psi-1case} and~\eqref{psi-2case} and using $r\ge21$, we obtain
	\begin{align*}
		&|\varphi_{r,12}(\theta)|\le \frac{21^6}{r^6} \prod_{k=1}^{12}\min\left\{1,\frac1{21\cdot |\sin(k\theta/2)|}\right\},\\
		&r^3|\varphi_{r,12}(\theta)|
		\le\frac{21^6}{r^3}\prod_{k=1}^{12}\min\left\{1,\frac1{21\cdot |\sin(k\theta/2)|}\right\}\le21^3\prod_{k=1}^{12}\min\left\{1,\frac1{21\cdot |\sin(k\theta/2)|}\right\}.
	\end{align*}
	This proves Claim~\ref{cla:convert}.
	
	To prove~\eqref{eq:r12-4},  by Claim~\ref{cla:convert}, it suffices to prove
	\begin{equation}\label{r=21}
		21^3 \cdot 12^{9/2}\int_{\pi/10}^{\pi}\theta^2\prod_{k=1}^{12}\min\left\{1,\frac1{21\cdot |\sin(k\theta/2)|}\right\}\,d\theta<4.
	\end{equation}
	\noindent\textbf{Proof of~\eqref{r=21}.}
	Let $\theta=2\pi s$.
	Then~\eqref{r=21} becomes
	\begin{equation*}
		21^3 \cdot 12^{9/2} 8\pi^3\int_{1/20}^{1/2}s^2\prod_{k=1}^{12}\min\left\{1,\frac1{21\cdot |\sin(\pi ks)|}\right\}\,ds<4.
	\end{equation*}
	We split the interval $[1/20,1/2]$ into the subintervals
	\[
	I_j=\left[\frac j{240},\frac{j+1}{240}\right],
	\qquad \textrm{for}\quad 12\le j\le119.
	\]
	Let
	\begin{equation*}
		P_j(s)=\prod_{k=1}^{12}q_{k,j}\quad\textrm{And}\quad q_{k,j}=\min\left\{1,\frac1{21|\sin(\pi ks)|}\right\},
	\end{equation*}
	where $q_{k,j}=1$ when $ks\in \mathbb{Z}.$
	Then 
	$$
	\int_{1/20}^{1/2}s^2\prod_{k=1}^{12}\min\left\{1,\frac1{21\cdot |\sin(\pi ks)|}\right\}\,ds=\sum_{j=12}^{119}
	\int_{I_j}s^2 P_j(s) \,ds
	$$
	We estimate the integral by finding
	a constant upper bound for $P_j(s)$ on each small interval $I_j$.
	
	Let 
	$$kI_j=\left[\frac{kj}{240},\frac{k(j+1)}{240}\right].$$
	If $kI_j$ contains an integer, then define $q'_{k,j}=1$.
	Thus $q_{k,j}\le q'_{k,j}$ for $s\in I_j.$
	
	Otherwise, $kI_j$ lies strictly between two consecutive
	integers $m$ and $m+1$. Define
	\[
	\delta_{k,j}=\min\left\{\frac{kj}{240}-m, \, m+1-\frac{k(j+1)}{240}\right\}>0.
	\]
	For every $s\in I_j$, we then have
	\begin{equation}\label{detkj}
		\delta_{k,j}\le ks-m\le1-\delta_{k,j}.
	\end{equation}
	Taking $a=\pi/2$ in~\eqref{eq:sine-chord},
	we have $\sin x\ge {2x}/{\pi}$.
	Thus, for $0\le u\le 1/2$, we have $\sin(\pi u)\ge 2u.$
	For $1/2<u\le 1$, by symmetry,
	$\sin(\pi u)=\sin(\pi (1-u))\ge 2(1-u).$
	We conclude that for $0\le u\le1$, 
	\begin{equation}\label{sinu2u}
		\sin(\pi u)\ge2\min\{u,1-u\}.
	\end{equation}
	Consequently, applying~\eqref{sinu2u} with $u=ks-m$,  we obtain
	\[
	|\sin(\pi ks)|
	=\sin\bigl(\pi(ks-m)\bigr)\ge 2\min\{ks-m,1-ks+m\}
	\ge2\delta_{k,j},
	\]
	where the last inequality follows by~\eqref{detkj}.
	In this case, we have 
	\[
	q'_{k,j}=\min\left\{1,\frac1{42\delta_{k,j}}\right\}\ge q_{k,j}.
	\]
	
	Multiplying these bounds over $k=1,\ldots,12$, we obtain  for $s\in I_j$,
	\[
	P_j(s)=\prod_{k=1}^{12}q_{k,j}\le \prod_{k=1}^{12}q'_{k,j}
	\]
	Since $s\le(j+1)/240$ on an interval of length $1/240$,
	\[
	\int_{I_j}s^2P_j(s)\,ds
	\le\frac{(j+1)^2}{240^3} \prod_{k=1}^{12}q'_{k,j}.
	\]
	
	Summing over all intervals and using
	$\pi<22/7$ and $\sqrt{12}<7/2$, we obtain
	\begin{align*}
		21^3 12^{9/2}(2\pi)^3
		\sum_{j=12}^{119}\int_{I_j}s^2P_j(s)\,ds&\le
		21^3 12^{9/2}(2\pi)^3
		\sum_{j=12}^{119}\frac{(j+1)^2}{240^3}
		\prod_{k=1}^{12}q'_{k,j}\\
		&<
		\frac{1509354}{125}
		\sum_{j=12}^{119}(j+1)^2
		\prod_{k=1}^{12}q'_{k,j}\\
        &<4.
	\end{align*}
	The final rational inequality is verified by the
	exact-arithmetic program in Appendix~\ref{app:fourier-code}.
	This proves~\eqref{r=21} and completes the proof.
\end{proof}

\begin{proposition}\label{lem:uniform-J2}
	For all integers $r\ge21$ and $n\ge12$,
	\begin{equation}\label{j392}
		J^{(3)}_{r,n}<\frac{6}{r^3n^{9/2}}.
	\end{equation}
\end{proposition}

\begin{proof}
	Recall that
	$$
	r^3n^{9/2}J^{(3)}_{r,n} = r^3n^{9/2} \int_{\pi/10}^{\pi} \theta^2|\varphi_{r,n}(\theta)|\,d\theta.
	$$
	By definition~\eqref{def-varphi},
	\begin{equation}\label{consec2}
		\varphi_{r,n+2}(\theta) = \varphi_{r,n}(\theta) \psi_r((n+1)\theta)\psi_r((n+2)\theta).
	\end{equation}
	The following claim gives an upper bound for $|\psi_r((n+1)\theta)\psi_r((n+2)\theta)|$.
	\begin{claim}\label{cla:pro49}
		Let $k$ be a positive integer. For ${\pi}/{10}\le\theta\le\pi,$ we have
		\begin{equation}\label{eq:adjacent-contraction}
			|\psi_r(k\theta)\psi_r((k+1)\theta)| <\frac{20}{43}.
		\end{equation}
	\end{claim}
	
	By~\eqref{consec2} and~\eqref{eq:adjacent-contraction}, we have 
	\begin{equation*}
		r^3(n+2)^{9/2}J^{(3)}_{r,n+2} \le \frac{20}{43}\left(\frac{n+2}{n}\right)^{9/2} 	r^3n^{9/2}J^{(3)}_{r,n}
		\le \frac{20}{43}\left(\frac76\right)^{9/2}	r^3n^{9/2}J^{(3)}_{r,n}<	r^3n^{9/2}J^{(3)}_{r,n}.
	\end{equation*}
	Thus, for each fixed $r\ge 21$, the sequence $r^3n^{9/2}J^{(3)}_{r,n}$ is nonincreasing along each parity
	class.
	Also, $|\psi_r|\le1$ gives
	$$
	r^313^{9/2} J^{(3)}_{r,13} \le\left(\frac{13}{12}\right)^{9/2} r^3 12^{9/2}J^{(3)}_{r,12}
	<\frac32 r^3 12^{9/2} J^{(3)}_{r,12}.
	$$
	By Lemma~\ref{lem:fixed-twelve}, for every $r\ge21$,
	\begin{equation*}
		r^3 12^{9/2} J^{(3)}_{r,12}=r^3\,12^{9/2} \int_{\pi/10}^{\pi}
		\theta^2|\varphi_{r,12}(\theta)|\,d\theta<4.
	\end{equation*}
	Thus,
	$r^3n^{9/2}J^{(3)}_{r,n}<6$, and this proves~\eqref{j392}.
	In the following, we prove Claim~\ref{cla:pro49}.

	\noindent\textbf{Proof of Claim~\ref{cla:pro49}.}
	Recall that by~\eqref{pvminimum}
	$$|\psi_r(k\theta)|\le \min\left\{1,\frac1{r\cdot |\sin(k\theta/2)|}\right\}
    =\frac1{\max\{1,r\cdot |\sin(k\theta/2)|\}}.$$
    Then
	\begin{align}\label{thelast}
	|\psi_r(k\theta)\psi_r((k+1)\theta)| &\le\frac1{\max\{1,r\cdot|\sin(k\theta/2)|\} \cdot \max\{1,r\cdot|\sin((k+1)\theta/2)|\}}\nonumber\\
	& \le\frac1{r\cdot|\sin(k\theta/2)|+r\cdot|\sin((k+1)\theta/2)|-1},
    \end{align}
	where the last inequality follows from $
	\max\{1,a\}\max\{1,b\}\ge a+b-1$.
	Since
	\[
	\begin{aligned}
		\sin({\theta}/{2})
		&=\sin((k+1)\theta/2-k\theta/2)\\
		&=\sin((k+1)\theta/2)\cdot \cos(k\theta/2)-\cos((k+1)\theta/2)\cdot\sin(k\theta/2)\\
		&\le |\sin((k+1)\theta/2)|+|\sin(k\theta/2)|,
	\end{aligned}
	\]
	combining~\eqref{thelast},
	we have 
	\begin{equation*}
	|\psi_r(k\theta)\psi_r((k+1)\theta)|\le  \frac{1}{r\sin({\theta}/{2})-1}.
    \end{equation*}
	Since $r\ge 21$ and $\pi/10\le \theta\le \pi$, we have 
	$$
	|\psi_r(k\theta)\psi_r((k+1)\theta)|\le \frac{1}{r \sin({\theta}/{2})-1}\le\frac{1}{21\sin(\pi/20)-1}<\frac{20}{43}.
	$$
	This proves Claim~\ref{cla:pro49} and completes the proof.
\end{proof}

\section{Proof of Theorem~\ref{thm:main}}
\begin{proof}[Proof of Theorem~\ref{thm:main}]
	Theorem~\ref{thm:main} holds for $2\le r\le20$ due to Hughes~\cite{Hug77} and Almkvist~\cite{Alm87,Alm89}.
	Now suppose that $r\ge21$.
	Theorem~\ref{low:seeds} gives the unimodality of $F_{r,n}$ for even $r\ge 22$ with all $1\le n\le 11$ and
	for odd $r\ge 21$ with $n=11$.
	Thus, in both cases, $F_{r,11}$ provides the starting point
	for induction.
	
	Fix an integer $r\ge21$. Let $n\ge12$, and assume that $F_{r,n-1}(q)$ is unimodal.
	By Theorem~\ref{checkinginterval},
	it suffices to prove that $B_{r,n}(j)\ge 0$ for every integer $j$ satisfying
	$$
	\lfloor\mu_{r,n-1}\rfloor< j\le \lfloor\mu_{r,n}\rfloor.
	$$
	Recall that, by~\eqref{fd:t},
	\begin{equation*}
		t=\mu_{r,n}-j+\frac12>0.
	\end{equation*}
	By~\eqref{bound:t}, we have 
	\begin{equation*}
		t\le\mu_{r,n}-\mu_{r,n-1}+\frac12=\frac{(r-1)n+1}{2}\le\frac{rn}{2}.
	\end{equation*}
	By~\eqref{eq:local-tail},~\eqref{3j1},~\eqref{3j2} and~\eqref{3j3}, we get 
	$$
	B_{r,n}(j)\ge\frac{2r^n}{\pi} \left( P_{r,n}(t)-\frac{t}{2}J_{r,n} \right)=\frac{2r^n}{\pi} \left(
	P_{r,n}(t)-\frac{t}{2}\left(J^{(1)}_{r,n}+J^{(2)}_{r,n}+J^{(3)}_{r,n}\right) \right).
	$$
	Proposition~\ref{lem:positive-lobe} gives
	\begin{equation*}
		P_{r,n}(t)>\frac{315\,t}{16\,r^3n^{9/2}}.
	\end{equation*}
	Propositions~\ref{lem:uniform-J0},~\ref{lem:large-n-tail} and~\ref{lem:uniform-J2}  give
	$$
	r^3n^{9/2}J_{r,n}^{(1)}<\frac{63}{2},\quad r^3n^{9/2}J_{r,n}^{(2)}<\frac{3}{2},\quad
	r^3n^{9/2}J_{r,n}^{(3)}<6.
	$$
	Thus
	$$
	P_{r,n}(t)-\frac{t}{2}J_{r,n}\ge
	\frac{315\,t}{16\,r^3n^{9/2}}-\frac{39\,t}{2\,r^3n^{9/2}}=\frac{3\,t}{16\,r^3n^{9/2}}>0.
	$$
	We obtain  the unimodality for all
	$n\ge 12$ by induction on $n$.
	This completes the proof.
\end{proof}

\par\medskip
\noindent\textbf{Declaration of competing interest.}\enspace
The authors declare that they have no known competing financial interests or personal relationships that could have appeared to influence the work reported in this paper.

\par\medskip
\noindent\textbf{Data availability.}\enspace
No external datasets were used in this work.
The source code for the exact computational verifications
is provided in Appendix~\ref{app:low-code}
and Appendix~\ref{app:fourier-code}.

\clearpage
\appendix

\section{Proof of~\eqref{eq:L-center-zero}}
\label{app:center-zero}

\textbf{Theorem }
Let $n\ge3$, and let $L(y)$ be defined by~\eqref{eq:Ly-sum}
on a subinterval $I=[a,b]$, where
$a<b=n(n+1)/4$ and $(a,b)$ contains no integer.
Then
\begin{equation*}
	L\left(\frac{n(n+1)}4\right)=0.
\end{equation*}
\begin{proof}
	Recall that by~\eqref{eq:unl},
	\begin{equation}\label{A1}
		\prod_{k=1}^{n}(1-q^k)=\sum_{\ell=0}^{2b}u_{n,\ell}q^\ell.
	\end{equation}
	By~\eqref{eq:Ly-sum},
	$$L(y)=\frac1{n!(n-2)!}
	\sum_{\ell\in\mathcal J_I}u_{n,\ell}(y-\ell)^{n-2}.$$
	
	We first prove that for every $x,$
	\begin{equation}\label{app:center-full-sum}
		\sum_{\ell=0}^{2b}u_{n,\ell}(x-\ell)^{n-2}=0.
	\end{equation}
	For a polynomial $f(x)$ and $1\le k\le n$,
	define the difference operator by
	\begin{align}\label{eq:diff-def}
		\nabla_k f(x)=f(x)-f(x-k).
	\end{align}
	If $\deg f(x)\ge 1$, then
	$\deg \nabla_k f(x)=\deg f(x)-1$.
	If $f(x)$ is constant, then $\nabla_k f(x)=0$.
	Thus, for $f(x)=x^{n-2}$, we have
	$$
	\nabla_1\cdots\nabla_n f(x)=0.
	$$
	It follows from~\eqref{eq:diff-def} that
	\begin{equation}\label{app:difference-expansion}
		\nabla_1\cdots\nabla_n f(x)
		=\sum_{S\subseteq\{1,\ldots,n\}}(-1)^{|S|}
		f\left(x-\sum_{k\in S}k\right).
	\end{equation}
	See~\cite[Proposition~1, p.~429]{SWMP17}.

	By definition, 
	\begin{align}\label{eq:Uq-S}
		\prod_{k=1}^{n}(1-q^k)
		=\sum_{S\subseteq\{1,\ldots,n\}}
		(-1)^{|S|}q^{\sum_{k\in S}k}.
	\end{align}
	Comparing the coefficients of $q^\ell$ in~\eqref{A1}
	and~\eqref{eq:Uq-S}, we obtain
	\begin{equation}\label{app:subset-coefficient}
		u_{n,\ell}
		=\sum_{\substack{S\subseteq\{1,\ldots,n\}\\
				\ell=\sum_{k\in S}k}}
		(-1)^{|S|}.
	\end{equation}
	By~\eqref{app:difference-expansion} and~\eqref{app:subset-coefficient}, 
	we get
	\begin{align*}
		0
		=\nabla_1\cdots\nabla_n f(x)
		=\sum_{\ell=0}^{2b}
		\left(
		\sum_{\substack{S\subseteq\{1,\ldots,n\}\\
				\sum_{k\in S}k=\ell}}
		(-1)^{|S|}
		\right)f(x-\ell)
		=\sum_{\ell=0}^{2b}u_{n,\ell}(x-\ell)^{n-2}.
	\end{align*}
	This proves~\eqref{app:center-full-sum}.
	Setting $x=b$ in~\eqref{app:center-full-sum} yields
	\begin{equation}\label{x=b}
		\sum_{\ell=0}^{2b}u_{n,\ell}(b-\ell)^{n-2}=0.
	\end{equation}

	Clearly
	\begin{align}\label{eq:U-sym}
		q^{2b} \prod_{k=1}^{n}(1-q^{-k})
		=(-1)^n\prod_{k=1}^{n}(1-q^k),
	\end{align}
	Comparing the coefficients of $q^\ell$ on both sides
	of~\eqref{eq:U-sym}, we obtain
	\begin{equation*}
		u_{n,2b-\ell}=(-1)^nu_{n,\ell}
	\end{equation*}
	for $0\le\ell\le 2b$.
	Thus, 
	\begin{align*}
		u_{n,2b-\ell}\bigl(\ell-b\bigr)^{n-2}
		=(-1)^nu_{n,\ell}(\ell-b)^{n-2}
		=u_{n,\ell}(b-\ell)^{n-2}.
	\end{align*}
	Hence, 
	\begin{align*}
		\sum_{\ell=0}^{2b}u_{n,\ell}(b-\ell)^{n-2}&= \sum_{\ell< b}
		u_{n,\ell}(b-\ell)^{n-2}+\sum_{\ell>b}
		u_{n,\ell}(b-\ell)^{n-2}\\
		&=\sum_{\ell<b}
		u_{n,\ell}(b-\ell)^{n-2}+\sum_{\ell>b}
		u_{n,2b-\ell}\bigl(\ell-b\bigr)^{n-2}\\
		&=2\sum_{\ell<b}
		u_{n,\ell}(b-\ell)^{n-2}.
	\end{align*}
	Since $\sum_{\ell=0}^{2b}u_{n,\ell}(b-\ell)^{n-2}=0$ by~\eqref{x=b},
	we obtain
	\begin{equation}\label{app:center-half-sum}
		\sum_{\ell<b}
		u_{n,\ell}(b-\ell)^{n-2}=0.
	\end{equation}
	
	Finally, since the interval $(a,b)$ contains no integer,
	the definition of $\mathcal J_I$ gives
	$$
	\mathcal J_I
	=\left\{\ell\in\mathbb Z_{\ge0}:
	\ell<\frac{a+b}{2}\right\}
	=\{\ell\in\mathbb Z_{\ge0}:\ell<b\}.
	$$
	Therefore, by~\eqref{eq:Ly-sum}
	and~\eqref{app:center-half-sum}, we conclude that
	$$
	L\left(\frac{n(n+1)}4\right)
	=L(b)
	=\frac{1}{n!(n-2)!}
	\sum_{\ell<b}
	u_{n,\ell}(b-\ell)^{n-2}
	=0.
	$$
	This completes the proof.
\end{proof}

\newpage
\section{Computer-assisted proof of~\eqref{low:noncentral-cert}
and~\eqref{low:central-cert}}
\label{app:low-code}

In this appendix, we give a self-contained Maple program
for the inequalities~\eqref{low:noncentral-cert}
and~\eqref{low:central-cert} in Proposition~\ref{sc-B}.
It treats even $r\ge22$ for $5\le n\le11$ and odd
$r\ge21$ for $n=11$.
In the program, \texttt{rparity} records $r'=r\bmod 2$:
it is $0$ for even $r$ and $1$ for odd $r$.
\vspace{0.1cm}

\noindent\textbf{Reconstructing the coefficient polynomials.}
For each $0\le s<M_n$, \texttt{CoefficientData} determines
$R_s$ from the $n-1$ values specified in
Lemma~\ref{lemma-key}.
It checks one additional value and the common leading coefficient.
The array \texttt{rcoef[s,k]} stores
$a_{k,s}=[x^k]R_s(x)$ from~\eqref{direct:Rs-expansion}.
All arithmetic used to check the inequalities is exact
integer or rational arithmetic.

\noindent\textbf{Removing repeated computations of $a_{k,s}$.}
We divide the computations of $a_{k,s}$ into two cases:
$1\le k\le n-3$ and $k=0$.

For the first case $1\le k\le n-3$,
by~\eqref{direct:H-periodicity}, we have shown that it suffices
to consider $21\le r\le20+M_n$ and $0\le j<M_n$.
The program sets
$$
p_{n,k}
=\operatorname{lcm}
\left(1,\ldots,\left\lfloor\frac{n}{k+1}\right\rfloor\right).
$$
To remove repeated computations, the program checks
\begin{equation*}
a_{k,s}=a_{k,\,s\bmod p_{n,k}}
\end{equation*}
for every $0\le s<M_n$.
It also checks $p_{n,k}\mid M_n$ and stops if either check fails.
These exact comparisons establish the shorter coefficient period.
Thus, for $k\ge1$, the maximum in~\eqref{direct:D-positive}
can be computed using $p_{n,k}$ consecutive values of each parameter.
For $p_{n,k}>1$, the program uses
$21\le r\le20+p_{n,k}$ with $r\equiv r'\pmod2$
and $0\le j<p_{n,k}$.
Such a period is even, so the required parity is retained.
For $p_{n,k}=1$, the coefficients are independent of $r,j$,
and one value of $r$ with the required parity suffices.
For $5\le n\le11$, all tested periods are at most $60$.

For the second case $k=0$, the program uses the period $M_n$.
It uses pairwise coprime periods $(12,5)$ for $n=5,6$,
$(12,5,7)$ for $n=7$, $(24,5,7)$ for $n=8$,
$(72,5,7)$ for $n=9$, $(360,7)$ for $n=10$,
and $(360,7,11)$ for $n=11$.
Their product is $M_n$ in each case.
For each period, the program forms a periodic array
by averaging $a_{0,s}$ over the corresponding residue classes.
From every array except the first, it subtracts
the mean of all $a_{0,s}$.
It then checks the resulting decomposition of $a_{0,s}$
for every $0\le s<M_n$.

By the Chinese remainder theorem, the residues of $r,j$
modulo these periods can be chosen independently.
The first period is even, so its residue of $r$
also determines the required parity.
The minimum and maximum of $H_{I,0}^{r,j}$
are therefore the sums of the component minima and maxima,
respectively.
If these two values are $H_{\min}$ and $H_{\max}$, then
$$
D_{I,0}^{r'}=\max\{|H_{\min}|,|H_{\max}|\}.
$$
This computes the value in~\eqref{direct:D-positive}
without changing its definition.

\noindent\textbf{Computing the upper bounds.}
For each fixed subinterval $I=[a,b]$, the coefficient
$h_{k,m}^{r,j}$ in~\eqref{direct:H-expansion} is
$$
h_{k,m}^{r,j}
=
\binom{k}{m}
\left(\frac{b-a}{2}\right)^m
\sum_{\ell\in\mathcal J_I}
u_{n,\ell}a_{k,s_\ell}
\left(\frac{a+b}{2}-\ell\right)^{k-m}
$$
for $0\le m\le k$.
For $k\ge1$, \texttt{DirectBounds} computes
$\sum_{m=0}^{k}|h_{k,m}^{r,j}|$
for each required pair of residue classes and takes the maximum.
For $k=0$, it uses the checked decomposition above.
The Maple array \texttt{D} stores the resulting
$D_{I,k}^{r'}$ from~\eqref{direct:D-positive}.

\noindent\textbf{Checking the final inequalities.}
For each subinterval, the program computes
\begin{equation}\label{app:direct-margin}
\Gamma_{I,r'}=
\begin{cases}
\displaystyle
E_{I,L}-\sum_{k=0}^{n-3}
D_{I,k}^{r'}21^{k+2-n},
& b\ne n(n+1)/4,\\[8pt]
\displaystyle
E_{[a,n(n+1)/4],\widehat L}/2
-\sum_{k=0}^{n-3}D_{I,k}^{r'}21^{k+3-n},
& b=n(n+1)/4.
\end{cases}
\end{equation}
In the program and its output, \texttt{margin} denotes
$\Gamma_{I,r'}$, the left-hand side of the inequality
being checked.
Positivity of~\eqref{app:direct-margin} proves
the corresponding inequality~\eqref{low:noncentral-cert}
or~\eqref{low:central-cert}.

All $56$ values exceed $1/100000$.
Their exact minimum is attained for $n=11$, $r'=1$,
and $I=[9,10]$.
The following table gives downward-rounded decimal lower bounds
for the minimum.

\begin{center}
\begin{tabular}{rrrrl}
\toprule
$n$ & $r'$ & Subintervals & Checks
& Lower bound for minimum\\
\midrule
5  & 0 & 3  & 3  & 0.010918997228\\
6  & 0 & 4  & 4  & 0.006182231490\\
7  & 0 & 4  & 4  & 0.003454538231\\
8  & 0 & 4  & 4  & 0.002046039076\\
9  & 0 & 5  & 5  & 0.001385491669\\
10 & 0 & 6  & 6  & 0.000896559893\\
11 & 0 & 6  & 6  & 0.000558864058\\
11 & 1 & 24 & 24 & 0.000010854568\\
\bottomrule
\end{tabular}
\end{center}

To run the program, paste the entire listing into a single
1-D Maple Input execution group in Worksheet mode,
save the worksheet as \texttt{Appendix\_B.mw}, and execute it.
Copy only the Maple code, without line numbers or LaTeX commands.
The code runs in a procedure with local variables and requires
no separate initialization or working-directory setting.
The worksheet displays $n$, $r'$, the interval endpoints,
the leading-term bound, all $D_{I,k}^{r'}$,
the subtracted sum, and $\Gamma_{I,r'}$ for each subinterval.
The last column, labelled \texttt{margin}, contains
the values of $\Gamma_{I,r'}$.
No external output file is created.
The program stops if a check fails.

\begin{lstlisting}[language=Maple]
proc()
local CoefficientData, IntervalLowerBound, DirectBounds,
VerifyRow, VerifyInitial;

CoefficientData := proc(n::posint)
option remember;
local Mn, Nn, limit, c, rcoef, periods, splitLists,
splitPeriods, parts, avg0, maxPeriod, h, t, p,
s, i, k, m, x, z, Rs, values, points;

if n < 5 or n > 11 then
error "n must lie between 5 and 11";
end if;

Mn := ilcm(seq(k,k=2..n));
Nn := n*(n+1)/2;
limit := n*Mn-1;

c := Array(0..limit,fill=0);
c[0] := 1;
for k from 2 to n do
for m from k to limit do
c[m] := c[m]+c[m-k];
end do;
end do;

rcoef := Array(0..Mn-1,0..n-2,fill=0);
points := [seq(i,i=0..n-2)];

for s from 0 to Mn-1 do
values := [seq(c[s+Mn*i],i=0..n-2)];
Rs := expand(subs(
z=(x-s-(Nn-1)/2)/Mn,
interp(points,values,z)));

for k from 0 to n-2 do
rcoef[s,k] := coeff(Rs,x,k);
end do;

m := s+(n-1)*Mn;
if subs(x=m+(Nn-1)/2,Rs) <> c[m] then
error "interpolation check failed",n,s;
end if;

if rcoef[s,n-2] <> 1/(n!*(n-2)!) then
error "leading coefficient check failed",n,s;
end if;
end do;

periods := Array(1..n-3,fill=1);
for k from 1 to n-3 do
periods[k] := ilcm(seq(i,i=1..iquo(n,k+1)));

if (Mn mod periods[k]) <> 0 then
error "period does not divide Mn",n,k;
end if;

for s from 0 to Mn-1 do
if rcoef[s,k]
<> rcoef[s mod periods[k],k] then
error "coefficient period check failed",n,k,s;
end if;
end do;
end do;

# Split the constant coefficient into coprime periods.
splitLists := [
[12,5], [12,5], [12,5,7],
[24,5,7], [72,5,7], [360,7], [360,7,11]
];
splitPeriods := splitLists[n-4];

if mul(splitPeriods[h],h=1..nops(splitPeriods)) <> Mn then
error "constant-period product failed",n;
end if;

for h from 1 to nops(splitPeriods) do
for t from h+1 to nops(splitPeriods) do
if ilcm(splitPeriods[h],splitPeriods[t])
<> splitPeriods[h]*splitPeriods[t] then
error "constant periods not coprime",n,h,t;
end if;
end do;
end do;

avg0 := add(rcoef[s,0],s=0..Mn-1)/Mn;
maxPeriod := max(
seq(splitPeriods[h],h=1..nops(splitPeriods)));

parts := Array(
1..nops(splitPeriods),0..maxPeriod-1,fill=0);

for h from 1 to nops(splitPeriods) do
p := splitPeriods[h];

for s from 0 to Mn-1 do
parts[h,(s mod p)] :=
parts[h,(s mod p)]+rcoef[s,0];
end do;

for t from 0 to p-1 do
parts[h,t] := p*parts[h,t]/Mn;
if h > 1 then
parts[h,t] := parts[h,t]-avg0;
end if;
end do;
end do;

for s from 0 to Mn-1 do
if rcoef[s,0] <>
add(parts[h,(s mod splitPeriods[h])],
h=1..nops(splitPeriods)) then
error "constant decomposition failed",n,s;
end if;
end do;

return [Mn,rcoef,periods,splitPeriods,parts];
end proc:

IntervalLowerBound := proc(
P,x::name,a::rational,b::rational)
local z,T,i;

if a >= b then
error "invalid interval",a,b;
end if;
if P = 0 then
return 0;
end if;

T := expand(subs(x=(a+b)/2+(b-a)*z/2,P));
return coeff(T,z,0)
-add(abs(coeff(T,z,i)),i=1..degree(T,z));
end proc:

DirectBounds := proc(
n::posint,rparity::nonnegint,
a::rational,b::rational,u,indices,data)
local rcoef,periods,splitPeriods,parts,D,k,m,
period,rvalues,r,j,i,h,minH,maxH,
componentMin,componentMax,value,support,
weights,transformed,loss,Nn;

rcoef,periods := data[2],data[3];
Nn := n*(n+1)/2;
D := Array(0..n-3,fill=0);

# Compute max |H_{I,0}^{r,j}| using coprime periods.
splitPeriods,parts := data[4],data[5];
support := select(i -> u[i]<>0,indices);
minH,maxH := 0,0;

for h from 1 to nops(splitPeriods) do
period := splitPeriods[h];

if h = 1 then
rvalues := select(
i -> (i mod 2)=rparity,
[seq(i,i=21..20+period)]);
else
rvalues := [seq(i,i=0..period-1)];
end if;

componentMin,componentMax := infinity,-infinity;

for r in rvalues do
for j from 0 to period-1 do
value := add(
u[i]*parts[h,((j-r*i) mod period)],
i in support);

componentMin := min(componentMin,value);
componentMax := max(componentMax,value);
end do;
end do;

minH := minH+componentMin;
maxH := maxH+componentMax;
end do;

D[0] := max(abs(minH),abs(maxH));

for k from 1 to n-3 do
period := periods[k];

if period = 1 then
rvalues := [22-rparity];
else
rvalues := select(
i -> (i mod 2)=rparity,
[seq(i,i=21..20+period)]);
end if;

# Coefficients of H((a+b)/2+(b-a)*z/2).
weights := Array(0..k,0..Nn,fill=0);

for m from 0 to k do
for i in indices do
weights[m,i] :=
u[i]*binomial(k,m)
*((a+b)/2-i)^(k-m)
*((b-a)/2)^m;
end do;
end do;

for r in rvalues do
for j from 0 to period-1 do
transformed := Array(0..k,fill=0);

for m from 0 to k do
transformed[m] := add(
weights[m,i]
*rcoef[(j-r*i) mod period,k],
i in indices);
end do;

loss := add(
abs(transformed[m]),m=0..k);
D[k] := max(D[k],loss);
end do;
end do;
end do;

return D;
end proc:

VerifyRow := proc(
n::posint,rparity::nonnegint,out)
local data,Nn,center,left,knots,nIntervals,expectedCounts,
a,b,i,k,interval,x,U,u,indices,Lpoly,lambda,D,
lowerLoss,margin,minimum,central,threshold;

if not ((rparity=0 and 5<=n and n<=11)
or (rparity=1 and n=11)) then
error "unsupported case",n,rparity;
end if;

data := CoefficientData(n);
Nn := n*(n+1)/2;
center := Nn/2;

if rparity = 0 then
left := n*(n-1)/4;
else
left := 9;
end if;

knots := sort(convert(
{left,center,
seq(i,i=ceil(left)..floor(center))},list));
nIntervals := nops(knots)-1;
expectedCounts := [3,4,4,4,5,6,6];

if (rparity=0 and nIntervals<>expectedCounts[n-4])
or (rparity=1 and nIntervals<>24) then
error "interval count failed",n,rparity;
end if;

U := expand(mul(1-x^i,i=1..n));
u := Array(0..Nn);
for i from 0 to Nn do
u[i] := coeff(U,x,i);
end do;

threshold := 1/100000;
minimum := infinity;

for interval from 1 to nIntervals do
a,b := knots[interval],knots[interval+1];
indices := select(
i -> i<(a+b)/2,[seq(i,i=0..Nn)]);

Lpoly := expand(
add(u[i]*(x-i)^(n-2),i in indices)
/(n!*(n-2)!));
central := evalb(b=center);

if central then
if subs(x=center,Lpoly) <> 0 then
error "center identity failed",n,a,b;
end if;

Lpoly := normal(Lpoly/(center-x));
if not type(Lpoly,polynom(rational,x)) then
error "center division failed",n,a,b;
end if;
end if;

lambda := IntervalLowerBound(Lpoly,x,a,b);
D := DirectBounds(
n,rparity,a,b,u,indices,data);

if central then
lowerLoss :=
add(D[k]*21^(k+3-n),k=0..n-3);
margin := lambda/2-lowerLoss;
else
lowerLoss :=
add(D[k]*21^(k+2-n),k=0..n-3);
margin := lambda-lowerLoss;
end if;

if margin <= threshold then
error "inequality failed",n,rparity,a,b,margin;
end if;

minimum := min(minimum,margin);

fprintf(out,
"%d %d %a %a %a %a %a %a %a\n",
n,rparity,a,b,central,lambda,
[seq(D[k],k=0..n-3)],lowerLoss,margin);
end do;

printf(
"n=%d parity=%d intervals=%d minimum=%a\n",
n,rparity,nIntervals,minimum);

return [nIntervals,minimum];
end proc:

VerifyInitial := proc()
local out,n,counts,intervals,minimum,expected;

out := 'terminal';
intervals,minimum := 0,infinity;
expected := 142723963570977949063
/13148746438822509182976000;

fprintf(out,
"n rprime a b final lambda D lowerLoss margin\n");

for n from 5 to 11 do
counts := VerifyRow(n,0,out);
intervals := intervals+counts[1];
minimum := min(minimum,counts[2]);
end do;

counts := VerifyRow(11,1,out);
intervals := intervals+counts[1];
minimum := min(minimum,counts[2]);

if intervals <> 56 then
error "total count failed",intervals;
end if;

if minimum <> expected then
error "reported minimum does not match",minimum;
end if;

fprintf(out,
"PASS inequalities=%d minimum=%a\n",
intervals,minimum);

printf(
"All %d inequalities exceed 1/100000.\n",
intervals);
printf("Global minimum = %a.\n",minimum);

return true;
end proc:

VerifyInitial():
end proc():
\end{lstlisting}

\newpage
\section{Exact verification of the twelve-factor bound}
\label{app:fourier-code}

We verify the final rational inequality in the proof of the
twelve-factor estimate. We use the same intervals
\[
I_j=\left[\frac j{240},\frac{j+1}{240}\right],
\qquad 12\le j\le119,
\]
and retain the notation $P_j(s)$, $q_{k,j}$, $q'_{k,j}$ and
$\delta_{k,j}$. As in the proof, the dependence of $q_{k,j}$ on $s$
is implicit. The program computes only the constant bounds
$q'_{k,j}$, which satisfy
\[
P_j(s)=\prod_{k=1}^{12}q_{k,j}
\le\prod_{k=1}^{12}q'_{k,j}\qquad(s\in I_j).
\]
For convenience, denote the rational upper sum by
\[
S=\frac{1509354}{125}
\sum_{j=12}^{119}(j+1)^2\prod_{k=1}^{12}q'_{k,j}.
\]
The coefficient is obtained from
\[
\frac{21^3\,12^{9/2}(2\pi)^3}{240^3}
<\frac{21^3\,12^4(7/2)(44/7)^3}{240^3}
=\frac{1509354}{125}.
\]
The Maple program below verifies $S<39/10$,
as required in the proof.

For each pair $(k,j)$, the closed interval $kI_j$ contains an integer
if and only if
\[
\left\lceil\frac{kj}{240}\right\rceil
\le\left\lfloor\frac{k(j+1)}{240}\right\rfloor.
\]
This criterion includes both endpoints. If it holds, the program sets
$q'_{k,j}=1$. Otherwise, it sets $m=\lfloor kj/240\rfloor$ and computes
\[
\delta_{k,j}=\min\left\{\frac{kj}{240}-m,
m+1-\frac{k(j+1)}{240}\right\},\qquad
q'_{k,j}=\min\left\{1,\frac1{42\delta_{k,j}}\right\}.
\]
These are precisely the two cases in the proof. In the program,
\texttt{q\_prime[k,j]} represents $q'_{k,j}$, and
\texttt{delta[k,j]} represents $\delta_{k,j}$ in the second case.
The variables \texttt{left} and \texttt{right} are the two endpoints
of $kI_j$, and \texttt{term} stores
$\prod_{k=1}^{12}q'_{k,j}$ after the inner loop.
All arithmetic, floor and ceiling operations, and comparisons
are exact; no floating-point or trigonometric evaluation is involved.

The following program is supplied as the Maple worksheet
\texttt{Appendix\_C.mw}.
Alternatively, paste the entire listing into a single
1-D Maple Input execution group in Worksheet mode and execute it.
Copy only the Maple code, without line numbers or LaTeX commands.
The code runs in a procedure with local variables and requires
no separate initialization.
The result is displayed in Maple; no output file is created.

\begin{lstlisting}[language=Maple]
	# Paste this complete block into one 1-D Maple Input group.
	proc()
	local q_prime, delta, S, j, k, left, right, m, term;
	q_prime := table():
	delta := table():
	S := 0:
	
	for j from 12 to 119 do
	term := 1:
	for k from 1 to 12 do
	left := k*j/240:
	right := k*(j+1)/240:
	if ceil(left) <= floor(right) then
	q_prime[k,j] := 1:
	else
	m := floor(left):
	delta[k,j] := min(left-m, m+1-right):
	q_prime[k,j] := min(1, 1/(42*delta[k,j])):
	end if:
	term := term*q_prime[k,j]:
	end do:
	S := S + (j+1)^2*term:
	end do:
	
	S := (1509354/125)*S:
	if S >= 39/10 then
	error "Verification failed":
	end if:
	printf("S < 39/10: %a\n", evalb(S < 39/10));
	end proc():
\end{lstlisting}

The program checks $S<39/10$ and prints
\texttt{S < 39/10: true}.
Consequently,
\begin{align*}
	&21^3\,12^{9/2}(2\pi)^3
	\sum_{j=12}^{119}\int_{I_j}s^2P_j(s)\,ds\\
	&\qquad\le21^3\,12^{9/2}(2\pi)^3
	\sum_{j=12}^{119}\frac{(j+1)^2}{240^3}
	\prod_{k=1}^{12}q'_{k,j}
	<S<\frac{39}{10}<4,
\end{align*}
which completes the required verification.
\end{document}